\documentclass[12pt]{amsart}
\usepackage{amssymb,latexsym,amsmath,amsthm,amscd}
\usepackage{setspace}
\usepackage[all]{xy} \xyoption{arc}
\usepackage[left=3cm,top=2cm,right=3cm,bottom = 2cm]{geometry}
\usepackage{graphicx}
\usepackage[dvipsnames]{xcolor}
\usepackage{mathtools}
\usepackage{charter}
\usepackage{hyperref}
\usepackage{makecell}
\usepackage[justification=centering]{caption}
\usepackage{tabularray}
\theoremstyle{plain}
\newtheorem{thm}{Theorem}[section]
\newtheorem{lem}[thm]{Lemma}
\newtheorem{prop}[thm]{Proposition}
\newtheorem{cor}[thm]{Corollary}

\theoremstyle{definition}
\newtheorem{dfn}[thm]{Definition}
\newtheorem{ex}[thm]{Example}

\theoremstyle{remark}
\newtheorem{rmk}[thm]{Remark}

\newcommand{\cB}{\mathcal{B}}

\newcommand{\cE}{\mathcal{E}}

\newcommand{\cO}{\mathcal{O}}

\newcommand{\cV}{\mathcal{V}}

\newcommand{\bD}{\mathbf{D}}

\newcommand{\frakl}{\mathfrak{l}}

\newcommand{\fraks}{\mathfrak{s}}

\DeclareMathOperator{\uhp}{\mathcal{H}}

\DeclareMathOperator{\Aut}{Aut}

\DeclareMathOperator{\Endo}{End}

\DeclareMathOperator{\im}{im}

\DeclareMathOperator{\GL}{GL}

\DeclareMathOperator{\SL}{SL}

\newcommand*{\df}{\mathrel{\vcenter{\baselineskip0.5ex \lineskiplimit0pt
                     \hbox{\scriptsize.}\hbox{\scriptsize.}}} =}

\providecommand{\twomat}[4]{\left(\begin{matrix}#1&#2\\#3&#4\end{matrix}\right)}
\providecommand{\stwomat}[4]{\left(\begin{smallmatrix}#1&#2\\#3&#4\end{smallmatrix}\right)}

\providecommand{\pseries}[2]{#1[\![ #2 ]\!]}

\newcommand{\CC}{\mathbf{C}}

\newcommand{\ZZ}{\mathbf{Z}}

\newcommand{\RR}{\mathbf{R}}

\renewcommand{\Im}{\operatorname{Im}}

\DeclareMathOperator{\Id}{Id}

\newcommand{\bone}{\mathbf{1}}

\DeclareMathOperator{\obackslash}{\backslash\hspace{-3pt}\backslash}
\DeclareMathOperator{\sl2}{\fraks\frakl_2}

\begin{document}
\title[VOA bundles on modular curves]{Vertex operator algebra bundles on modular curves and their associated modular forms}
\author[Barake, Chuchman, Franc, Mason, Nasserden]{Daniel Barake, Owen Chuchman, Cameron Franc, Geoffrey Mason and Brett Nasserden}
\date{}

\begin{abstract}
This paper describes the vector bundle on the elliptic modular curve that is associated to a vertex operator algebra $V$ (VOA) or more generally a quasi-vertex operator algebra (QVOA), with a view towards future applications aimed at studying the characters of VOAs.\ We explain how the modes of sections of $V$ give rise naturally to $V$-valued quasi-modular forms.\ The space $Q(V)$ of $V$-valued quasi-modular forms is endowed with the structure of a doubled QVOA, and in particular the algebra $Q$ of quasi-modular forms is itself a doubled QVOA. $Q(V)$ also admits a natural derivative operator arising from the connection on the bundle defined by $V$ and the modular derivative, which we call the raising operator.\ We introduce an associated lowering operator $\Lambda$ on $Q(V)$ having the property that the $V$-valued modular forms $M(V)\subseteq Q(V)$ are the kernel of $\Lambda$. This extends the classical theory of scalar-valued quasi-modular forms.\ We exhibit an explicit isomorphism of $M(V)$ with $M \otimes V$. Finally, the coordinate invariance of vertex operators implies that $M(V)$ has a natural Hecke theory, and we use this isomorphism to fully describe the Hecke eigensystems: they are the same as the systems of eigenvalues that arise from scalar-valued quasi-modular forms.
\end{abstract}
\maketitle
\tableofcontents

\section{Introduction}\label{ss:notation}
Vertex operators arose initially in physics as an algebraic device for aiding computations in conformal field theory and related fields such as statistical mechanics and string theory. See \cite{CFTBook} for references to the physics literature. From the very start, surprising connections with modular forms were a part of these investigations. Later, mathematicians began to formalize and expand  the nascent theory of vertex operator algebras or VOAs, again mainly from an algebraic viewpoint \cite{FHL}.\ The main drivers of these investigations were to explain monstrous moonshine, and to explain a number of interesting phenomenon in the study of representations of some infinite dimensional Lie algebras. Inevitably, modular forms and quasi-modular forms often made contact with these investigations in surprising ways \cite{KacPetersen}, \cite{FLM}.

Since these initial investigations, some of the modular mystery has been swept away, largely due to the geometric interpretation of VOAs and the theories that inspired their study. Indeed, in \cite{FBZ}, one finds a clear discussion of how the data of a VOA is the same as the data of a certain vector bundle --- generally of infinite rank --- on any smooth curve, or even better, on moduli spaces of curves. These bundles are endowed with connections that interact with the vertex operations in interesting ways. Since modular forms are sections of vector bundles on the moduli space of elliptic curves, this geometric viewpoint makes for a potent and convenient way to understand the appearance of modular forms in the theory of VOAs.

Inspired by earlier $p$-adic work on VOAs \cite{FMpadic}, we began to investigate in what sense VOAs might also give rise to vector bundles on arithmetic curves such as, for example, the Igusa tower or other limits of Shimura curves. It became clear that it would be beneficial to have a detailed understanding of the geometric theory of VOAs on modular curves at finite level. The present paper constitutes our attempt to make this theory precise and accessible to readers already familiar with the theory of modular forms. More precisely, we give below a clear construction of the vector bundle $\cV/X(1)$ on the modular curve $X(1)$ of elliptic curves of level one attached to a VOA or, more generally, a quasi-VOA (QVOA) $V$, using an approach based on uniformization as presented, say, in the lectures notes \cite{Hain}. This approach allows us to describe the sections of $\cV$ in a clear way. These sections are holomorphic functions on the complex upper-half plane $\uhp$
\[
f\colon \uhp \to V
\]
taking values in a finite-dimensional subspace of $V$ and which satisfy a modular transformation law. Such functions, of varying weights, together describe a module $M(V)$ of $V$-valued modular forms. See Section \ref{s:VOAmodular} for precise definitions of $V$-valued modular forms and the space $M(V)$. 

Before turning to a discussion of our results, let us remark upon why we consider quasi-VOAs as a fundamental object of study below, as opposed to VOAs. The first basic point is that we are working on the moduli space of elliptic curves rather than the general genus case, so only the $\mathfrak{s}\mathfrak{l}_2$ part of the Virasoro action on $V$ really plays a r\^{o}le in our constructions. This observation has appeared in various guises in the literature, see for example \cite[Page 20]{FHL}. The second basic point is that the space $Q(V)$ of $V$-valued quasimodular forms that we introduce below, only has the structure of a quasi-VOA, and not a full-blown VOA. Thus, in this subject quasi-VOAs are unavoidable and a natural object of study.

In this paper we give a rather complete description of $M(V)$ and various related structures. To describe our results, it turns out to be natural and convenient to introduce a more general space of quasi-modular $V$-valued forms $Q(V)$, cf. Section \ref{ss:qmfdefinition}.\ The reason for introducing $V$-valued quasi-modular forms is simple:\ the space $M(V)$ is not closed under the natural pointwise bilinear products that $M(V)$ inherits from $V$.\ This assertion is precisely the VOA analog of the statement that 
the algebra of modular forms $M$ is not closed with respect to the operator 
$\theta=qd/dq$ and that one should expand to the algebra of quasi-modular forms $Q$ 
to get a good $\theta$-invariant space of functions. Although we introduce $Q(V)$ as a space of functions satisfying certain quasi-modular transformation laws, we prove in Theorem \ref{p:QMVtensor} that there is a natural graded isomorphism:
\begin{equation}
\label{eq:tensor}
Q(V) \cong Q\otimes_{\CC} V^{(2)},    
\end{equation}
where $V^{(2)}$ denotes the same underlying vector space as $V$, but with the grading doubled. This space $Q(V)$ is then closed under the natural bilinear products that $Q(V)$ inherits from $V$. Moreover, we explain in Section \ref{s:QVOA} how $Q(V)$ inherits a natural structure of a (doubled) quasi-vertex operator algebra from the identification of equation \eqref{eq:tensor}.\ This quasi-VOA structure includes, in particular, raising and lowering operators defining a representation of the Lie algebra $\mathfrak{sl}_2$.\ 

It is an irony that here we have quasi-modular forms and vertex algebras yoked together in a setting in which Rankin-Cohen (RC) brackets  appear to play no r\^{o}le, whereas it has long been anticipated \cite{ZagierIndian} that RC brackets of (quasi)modular forms and vertex algebras should be connected.\ One might approach this by appropriately performing a deformation quantization of $Q$ or $Q(V)$ considered as a QVOA in order to realize such a connection.\ Such things have been well-studied in the literature on modular forms e.g., \cite{Choieetal},\cite{ConnesMovsati}  nor  are they entirely unknown in the VOA literature \cite{EtingofKazhdan}, \cite{Lih-adic}, and we hope to return to this subject in the future.

Our first main result characterizes the containment of $V$-valued modular forms $M(V) \subseteq Q(V)$ inside the space of $V$-valued quasi-modular forms using the apparatus described so far. Recall that every quasi-modular form can be realized as a polynomial in the Eisenstein series $E_2,E_4,E_6$, (see \cite{RoyerQuasi}) making $\tfrac{\partial}{\partial E_2}$ a natural weight lowering linear operator on $Q$. On the other hand, $L(1)$ is a natural weight lowering operator on $V$. Combining these lowering operators gives:  
\begin{thm}
    \label{mainthm1}
    The space of modular forms $M(V)$ is the kernel of the lowering operator 
    \[
    \Lambda \df \frac{12}{2\pi i } \frac{\partial}{\partial E_2}\otimes 1 + 1\otimes L(1)
    \]
    acting on $Q(V)$.
\end{thm}

Theorem \ref{mainthm1} is proved as a consequence, cf.\ Corollary \ref{cokerL}, of Theorem \ref{thmXLambda} which describes the modular transformation behavior of elements of $Q(V)$.\ There are several perspectives that one can take with Theorem \ref{mainthm1}.\ First, on the modular side, it is a natural analog of a result in the classical theory of modular forms, namely that the subspace $M = \CC[E_4,E_6]$ of modular forms inside the space $Q= \CC[E_2,E_4,E_6]$ of quasi-modular forms is the kernel of the lowering operator $\partial/\partial E_2$. While this scalar-valued result is obvious once one knows the descriptions of $M$ and $Q$ in terms of Eisenstein series, Theorem \ref{mainthm1} requires more work.\ On the other hand it is well-known that the \textit{quasi-primary} states $QP(V)$ of a VOA $V$ are in some sense the most tractable (cf.\ \cite{FHL}, Section 2.6), and for VOAs $V$ we have (by definition) $QP(V)=\ker L(1)$.\ Thus a VOA-centric take on Theorem \ref{mainthm1} is that $M(V)$ may be viewed as the space of quasiprimary states in $Q(V)$.

It turns out that one can give a very concrete description of $M(V)$ as a free-module of infinite rank over $M$. This is our second main result:
\begin{thm}
    \label{mainthm2}
    There exists an explicit $M$-linear isomorphism of graded spaces:
    \[
    P\colon  M\otimes_{\CC} V^{(2)} \cong M(V)
    \]
\end{thm}
A subtlety is that the isomorphism $P^{-1}$ arising from Theorem \ref{mainthm2} is \emph{not} the restriction of the pointwise isomorphism of equation \eqref{eq:tensor} to the subspace $M(V) \subseteq Q(V)$. See Section \ref{SSP} for the definition of $P$ and a proof of Theorem \ref{mainthm2}.

Taken together, the isomorphism of equation \eqref{eq:tensor}, and Theorems \ref{mainthm1} and \ref{mainthm2} yield a detailed understanding of the space $M(V)$. Since the sections of a bundle characterize the bundle itself, this in turns provides a concrete, automorphic, description of the bundle $\cV/X(1)$ over the moduli space of elliptic curves that is associated to a QVOA or VOA $V$ as in, say, \cite{FBZ}.

Let us now outline the structure of this paper. In Section \ref{s:modforms} we recall the geometric definition of modular forms. Section \ref{s:VOAmodular} then describes the construction of the bundle $\cV/X(1)$ from \cite{FBZ} applied to a QVOA on the moduli space of elliptic curves following the orbifold approach of \cite{Hain}. Section \ref{s:connection} follows by introducing the connection on the bundle $\cV/X(1)$. This is a raising operator in the modular theory, and we provide detailed first-principle proofs to help ease the uninitiated into the subject. Then in Section \ref{s:qmf} we introduce quasi-modular forms and give some motivation for their consideration. In the closely related Section \ref{s:QVOA} we discuss the general algebraic structure of $Q(V)$ as a QVOA. Then in Section \ref{s:operators} we introduce the operators needed for proving our main Theorems, and we discuss some examples, dimension computations, and describe how the connection behaves when one applies the trivialization $P$ to the space $M(V)$. Finally, in Section \ref{s:hecke} we discuss certain natural Hecke operators on $M(V)$ that arise from the coordinate-invariance of the QVOA-bundle construction. We use our isomorphism $P$ from Theorem \ref{mainthm2} to show that the Hecke eigensystems that arise from $M(V)$ and $Q(V)$ are the same as those that arise from $M$ and $Q$, respectively, with appropriate multiplicities. 

We also include two appendices. The first, Appendix \ref{AppA}, summarizes some basic results on quasi-modular forms. It also include a discussion of Hecke operators on quasi-modular forms and explains how the Hecke eigensystems that one finds in $Q$ can be obtained easily from the Hecke eigensystems of modular forms of level one by differentiation. This result is known to experts, but we could not find a convenient reference and we include a proof for completeness. This facilitates the  determination of all Hecke eigenforms in the spaces $M(V)$ and $Q(V)$ from the Hecke eigenforms in $M$. Our second and final Appendix \ref{AppB} summarizes some background on vertex algebras, in particular highlighting some of the important background on quasi-VOAs.

\subsection{Notation}\label{SNotation} We list some commonly occurring notation and terminology here.
\subsubsection{General notation}
\begin{itemize}
    \item[---] $\GL_2(\RR)^+=$ group of real $2\times 2$ matrices having positive determinant.
    \item[---] $\Gamma\df\SL_2(\ZZ)$.
    \item[---] $S\df\stwomat{0}{-1}{1}{0}$ and $T\df \stwomat 1101$ are the standard generators of $\Gamma$.
    \item[---] $\mathfrak{sl}_2$ denotes the Lie algebra of traceless $2\times 2$ complex matrices.
    \item[---] $\uhp \df \{\tau\in\CC\mid \Im\tau>0\}$ is the complex upper half-plane.
    \item[---] $q\df e^{2\pi i\tau}$.
\end{itemize}

\subsubsection{Vertex algebras}
For further background see Appendix \ref{AppB}.\ $V=(V, Y, \mathbf{1})$ denotes a $\ZZ$-graded vertex algebra (Fock space); its decomposition into homogeneous pieces is denoted $V=\oplus _{n\in\ZZ} V_n$.\  If $V$ is of CFT-type we write this as
\[
V=\CC\bone \oplus V_1\oplus V_2\oplus \hdots
\]
\begin{itemize}
    \item[---]Elements $v\in V$ are called \emph{states} and elements $v\in V_n$ are the states of  \emph{conformal degree} $n$.
    \item[---] The vertex operator for $v \in V$ is $Y(v, z)\df \sum_{n\in\ZZ} v(n)z^{-n-1}$ and $v(n)\in \Endo(V)$ is the $n^{th}$ \emph{mode} of $v$.
    \item[---] The \emph{doubled} vertex algebra $V^{(2)}$ has the same Fock space $V$ but its grading is doubled: $V^{(2)}=\oplus_{n\in\ZZ} V^{(2)}_{2n}$, where $V^{(2)}_{2n}\df V_n$.
    \item[---] $F_NV\df \oplus_{n\leq N} V_n$ denotes the $N^{th}$ truncation of $V$.
    \item[---] If $V$ is a quasi-vertex algebra, $L(-1)$, $L(1)$, $L(0)$ denote the raising, lowering and Euler operators, respectively, for the canonical $\mathfrak{sl}_2$-action on $V$.
\end{itemize} 

\subsubsection{Cocyles and actions}
For $\gamma=\stwomat {a}{b}{c}{d} \in \GL_2(\RR)^+$,
\begin{itemize}
    \item[---] $j(\gamma, \tau)\df c\tau+d$.
    \item[---] $X=X(\gamma)=X(\gamma, \tau)\df\frac{c}{c\tau+d}$.
    \item[---] The \emph{canonical} $1$-cocycle for the action of $GL_2(\RR)^+$ on the trivial quasi-vertex operator algebra bundle $\uhp\times V$ (taking values in $\Endo(V)$) is
\[ K(\gamma,\tau)\df\exp\left(-c\det\gamma^{-1} j(\gamma,\tau)L(1)\right)j(\gamma,\tau)^{-2L(0)}(\det\gamma)^{L(0)}.\]
\end{itemize}
For meromorphic functions $f:\uhp \rightarrow U$ and integers $k$, where $U$ is a finite-dimensional complex vector space,
\begin{itemize}
    \item[---] $f|_k\gamma(\tau)\df j(\gamma, \tau)^{-k}(\det\gamma)^{k/2}f(\gamma\tau)$.
\item[---] if $U$ is contained in a quasi-vertex operator algebra, then \[f||_k \gamma(\tau)\df j(\gamma, \tau)^{-k}K(\gamma, \tau)^{-1}f(\gamma\tau).\]
\end{itemize}
\subsubsection{Modular forms}
\begin{itemize}
\item[---] Bernoulli numbers $B_k$ are defined by the generating series
$\tfrac{t}{e^t-1}=:\sum_{k\geq0} \tfrac{B_k}{k!}t^k.$
    \item[---] For integers $k\geq 1$, $E_{2k}\df  1-\frac{4k}{B_{2k}}\sum_{n\geq1} \sigma_{2k-1}(n)q^n$, is the weight $2k$ Eisenstein series for $\Gamma$ with constant term $1$.
    \item[---] In particular,
$E_2(\tau) = 1-24\sum_{n\geq1} \sigma_1(n)q^n$ is a quasi-modular form.
\item[---] $M=\CC[E_4, E_6]=\oplus_k M_{2k}$ is the graded algebra of holomorphic modular forms on $\Gamma$.
\item[---] $Q=\CC[E_2, E_4, E_6]=\oplus_k Q_{2k}$ is the graded algebra of holomorphic quasi-modular forms on $\Gamma$ (cf. Appendix \ref{AppA}).
\item[---] The half-graded algebra of quasi-modular forms is $Q^{(1/2)} = \oplus_{k\geq0} Q^{(1/2)}_k$ where $Q^{(1/2)}_k\df Q_{2k}$.
\end{itemize}

\section{Modular forms}
\label{s:modforms}
For the basics of vector bundles on the orbifold defined by the moduli space of elliptic curves, see \cite{Hain}.

\subsection{Geometric definition of modular forms}
Let $Y(1)$ be the (open) modular curve of level one associated to  $\Gamma$, and let $p\colon \cE \to Y(1)$ be the universal elliptic curve over $Y(1)$. We recall how this is defined:\ define left and right actions of $\Gamma$ and $\ZZ^2$ on $\uhp\times \CC$ as follows:
\begin{align*}
    \stwomat abcd (\tau,z) &= (\tfrac{a\tau+b}{c\tau+d},(c\tau+d)^{-1}z),\\
    (\tau,z) \cdot (m,n) &= (\tau, z+m\tau + n).
\end{align*}

The quotient $E\df (\uhp\times \CC) / \ZZ^2$ is a family of elliptic curves living over $\uhp$ by projection onto the first factor. The fiber over $\tau \in \uhp$ is $E_\tau \df \CC/(\ZZ\tau \oplus\ZZ)$. The double quotient
\[
  \cE = \Gamma \backslash E
\]
is then the universal elliptic curve over $Y(1) = \Gamma \obackslash \uhp$.
\begin{dfn}
Set $\omega_{\cE/Y(1)} = p_*\Omega^1_{\cE/Y(1)}$, an invertible sheaf on $Y(1)$.
\end{dfn}
The importance of this definition is that, under a uniformization, global sections of $\omega_{\cE/Y(1)}$ are holomorphic modular forms of weight one on the upper half plane, with no conditions imposed at cusps. To make this precise, note that the sheaf $\Omega^1_{\cE/Y(1)}$ can be described concretely by working on the cover $\uhp\times \CC /\uhp$ and describing $\Omega^1_{\uhp\times \CC/\uhp}$: sections of $\Omega^1_{\uhp\times \CC}$ are nothing but differentials $\alpha d\tau + \beta dz$ where $\alpha,\beta$ are holomorphic functions on $\uhp \times \CC$. Under this description, then, $\Omega^1_{\uhp\times \CC/\uhp}$ is the line bundle on $\uhp \times \CC$ spanned by $dz$. That is,
\[
\Omega^1_{\uhp \times \CC/\uhp}(\uhp\times \CC) = \{\alpha(\tau,z) dz \mid \alpha \textrm{ is holomorphic on } \uhp \times \CC\}.
\]
Then, since $E$ is the quotient of $\uhp \times \CC$ by $\ZZ^2$, we can obtain a description of the global sections of $\Omega^1_{E/\uhp}$ by taking invariants under this action:
\begin{align*}
    \Omega^1_{E/\uhp}(E) &=\Omega^1_{\uhp \times \CC/\uhp}(\uhp\times \CC)^{\ZZ^2}\\
    &=\{\alpha(z,\tau) dz \mid \alpha(\tau,z+m\tau+n) = \alpha(\tau,z) \textrm{ for all } m,n \in \ZZ\}.
\end{align*}
This is one of the axioms for a Jacobi form.

\begin{rmk}
    If we were working on $\Omega^1_{\uhp \times \CC}$, then $\ZZ^2$ would not act trivially on $dz$: we would have $dz \cdot (m,n) = dz + md\tau$, but when working in the relative differentials $\Omega^1_{\uhp \times \CC/\uhp}$, the $d\tau$ term vanishes.
\end{rmk}

We can continue this analysis by taking the quotient on the left by the modular group $\Gamma$. Invariance under this action amounts to the functional transformation law:
\[
  \alpha\left(\frac{a\tau+b}{c\tau+d}, (c\tau+d)^{-1}z\right) \frac{dz}{c\tau+d} = \alpha(\tau,z)dz
\]
That is, $\alpha(\tau,z)$ is a weak Jacobi form of weight $1$ and index $0$, where weak here means that we say nothing about its Fourier development:
\[
\Omega^1_{\cE/Y(1)}(\cE) = \{\alpha(\tau,z)dz \mid \alpha \textrm{ is a weak Jacobi form of weight } 1 \textrm{ and index } 0.\}
\]

Now, the projection map $p\colon \cE \to Y(1)$ is described concretely in terms of our uniformization as $p(\tau,z) = \tau$. It remains to identify the space
\[
\omega_{\cE/Y(1)}(Y(1)) = (p_*\Omega^1_{\cE/Y(1)})(Y(1))
\]
with the space of weakly holomorphic modular forms of weight $1$. By definition, we have $f \in \omega_{\cE/Y(1)}(Y(1))$ if and only if $f\circ p \in \Omega^1_{\cE/Y(1)}(Y(1))$. Assume first that $f\circ p \in \Omega^1_{\cE/Y(1)}(Y(1))$, so that
\begin{align*}
  f\left(\frac{a\tau+b}{c\tau+d}\right) &= (f\circ p)\left(\frac{a\tau+b}{c\tau+d},(c\tau+d)^{-1}z\right) \\
  &=(c\tau+d)(f\circ p)\left(\tau,z\right)\\
  &=(c\tau+d)f(\tau),
\end{align*}
where we have used the Jacobi form transformation law satisfied by $f\circ p$. Therefore every section $f \in \omega_{\cE/Y(1)}(Y(1))$ can be interpreted as a modular form of weight one in this way.

Suppose conversely that $f$ is a (weakly holomorphic) modular form of weight one. Then $f\circ p$ is independent of $z$, and it is clear that we obtain a Jacobi form $f\circ p$ of weight $1$ and index $0$ from this constancy in $z$ and the transformation law for $f$. Thus, we have indeed verified that global sections of $\omega_{\cE/Y(1)}$ can be identified with (weakly holomorphic) modular forms of weight one. Modular forms of weight $k$ are then sections of $\omega_{\cE/Y(1)}^k$.

\subsection{The compactified modular curve}
\label{ss:cusps}
Let now $X(1)$ denote the compact modular curve of level one with the cusp added. Let $\bD$ denote the unit disc, and let $\bD^\times$ denote the punctured unit disc. Endow both of these spaces with the trivial action by the cyclic group $C_2$ of order $2$. Then the compactified modular curve $X(1)$ is defined concretely, as an orbifold, by glueing $Y(1)$ and $C_2\obackslash \bD$ under the diagram
\[\xymatrix{
& \langle -1,T\rangle \obackslash \uhp \ar[dl]_{\iota_1}\ar[r]^{\tau \mapsto q} & C_2\obackslash \bD^\times \ar[dr]^{\iota_2}&\\
\Gamma\obackslash \uhp&&&C_2\obackslash \bD
}\]
where the left diagonal map is the natural covering and the rightmost diagonal map is the natural inclusion.

\medskip
In this way, a vector bundle $\cV$ on $X(1)$ can be described concretely as a triple $\cV = (\cV_1,\cV_2,\phi)$ where $\cV_1$ is a vector bundle on $Y(1)$, and $\cV_2$ is a vector bundle on $C_2\obackslash \bD$, and $\phi$ is a bundle isomorphism
\[
\phi \colon \iota_1^*\cV_1 \to \iota_2^*\cV_2.
\]
See \cite{Hain} for the case of line bundles, and Sections 2 and 3 of \cite{CandeloriFranc} for the higher rank case.

In our setting, we will have a QVOA $V$, or some finite-dimensional truncation of it, and bundles
\begin{align*}
\cV_1 &= \Gamma \obackslash (\uhp \times V),\\
\cV_2 &= C_2 \obackslash (\bD \times V)
\end{align*}
where this notation means that $\uhp\times V$ and $\bD\times V$ are endowed with actions lifting the actions of $\Gamma(1)$ and $C_2$ on $\uhp$ and $\bD$, respectively. Then
\begin{align*}
    \iota_1^*\cV_1 &= \langle -1, T\rangle \obackslash (\uhp \times V),\\
    \iota_2^*\cV_2 &= C_2 \obackslash (\bD^\times\times V)
\end{align*}
Thus, $\iota_1^*\cV_1$ is nothing but the trivial bundle $\uhp\times V$ endowed with the restriction of the action of $\Gamma(1)$ to $\langle -1, T\rangle$, and $\iota_2^*\cV_2$ is nothing but $\bD^\times \times V$ with the restricted action of $C_2$. Then $\phi$ is a map of the form
\[
\phi(\tau,v) = (q,\phi_0(\tau,v))
\]
that intertwines these two actions.

\medskip
It turns out that the action of $\Gamma$ on $V$ that intervenes below leads to a trivial action of $\langle -1,T\rangle$, and similarly a trivial action of $C_2$ on $V$. Therefore, there is a canonical way to glue these bundles, by taking $\phi_0(\tau,v) = v$. In this way one obtains a concrete description of QVOA bundles on compact modular curves. When one changes coordinates from $\uhp$ to the coordinates provided by the $q$-disc 
$\bD$, the square-bracket formalism of \cite{Zhu} (see also \cite{DLMmodular, MT}) arises naturally via Huang's change of coordinates formula, see Lemma 5.4.6 of \cite{FBZ}. Hence, in this sense, the bundle associated to $V$ on the noncompact curve $Y(1)$ is completed to a bundle on the compact curve $X(1)$ by gluing in $V$, under the square-bracket formalism, at the cusp.\ This viewpoint will not play a strong r\^{o}le in what follows, as we will define modular forms using a more ``trivial'' notion of extension to the cusp, via $q$-expansions. We will point out how this trivial notion of holomorphy at the cusp differs from the natural one arising from coordinate- invariance of the VOA bundle construction.

\section{VOAs on modular curves}

\subsection{The main construction}
The goal here is to give a similarly concrete description of how a quasi-vertex operator algebra $V$ is described geometrically as a family of vector bundles on all modular curves. For simplicity we shall restrict to level one. This will allow us to give a simplified presentation that involves, in particular, a unique cusp.

Let $\cV_{Y(1)}$ denote the bundle associated to $V$ on $Y(1)$ as described in \cite{FBZ}. We wish to describe this as in the previous section, as a trivial bundle on $\uhp$ endowed with an action of $\Gamma$ such that
\[
\cV_{Y(1)} \df \Gamma \obackslash (\uhp \times V).
\]
That is, to begin with, we consider the trivial bundle $\uhp \times V$, but we must factor out the action of the modular group that arises by change of coordinates under linear fractional transformations. 

\subsubsection{The torsor of local coordinates} Following \cite{FBZ}, let $\cO$ denote a power series ring over $\CC$ in one variable, and let $\Aut_{\cO}$ denote its group of automorphisms. As a set we identify
\[
\Aut_{\cO} \df \left\{\sum_{n\geq 1} a_nX^n \mid  a_1\neq 0\right\},
\]
and multiplication is defined as $(f*g)(X) = g(f(X))$. 

We begin with the covering torsor
\[\Aut_{\uhp} \to \uhp\] 
of local coordinates. Let $z$ denote the usual variable on $\CC$, which is a local coordinate at $z=0$ in $\CC$. Restrict this global variable to $\uhp$. Notice that $z$ is not a local coordinate at any point of $\uhp$. However, at a point $\tau \in \uhp$, our global variable $z$ on $\uhp$ yields the local coordinate $X=z-\tau$ at $\tau$. With this notation we trivialize the torsor of local coordinates via the identification
\[
\Aut_{\uhp} = \uhp \times \Aut_{\cO}.
\]

We now describe an action of $\GL_2^+(\RR)$ on this bundle. Given 
\[f(z) = \sum_{n\geq 1} a_n(z-\tau)^n\] 
a local coordinate at $\tau$ in terms of the global variable $z$, then
\[
f(\gamma^{-1}z) = \sum_{n\geq 1} a_n(\gamma^{-1}z-\tau)^n
\]
is a local coordinate at $\gamma \tau$. This endows $\Aut_{\uhp}$ with a natural left action of $\GL_2^+(\RR)$ that lifts the action on $\uhp$:
\[
\gamma (\tau,f(z)) \df (\gamma \tau, f(\gamma^{-1} z)).
\]
This provides the extra data needed to descend to a bundle on $Y(1)$. 

\begin{rmk}
    The added flexibility of being able to change coordinates also by the matrices in $\GL_2^+(\RR)$ that are not in the modular group will be used below to define Hecke operators. 
\end{rmk}

To make this more concrete, we will identify all fibers of $\Aut_{\uhp}$ with one another by setting $X = z-\tau$. Now set
\[
\Aut_{Y(1)} \df \uhp \times \Aut_\cO.
\]
This bundle is endowed with the following left action of $\GL_2^+(\RR)$, which is the transport of structure of the action above under this reparameterization of the fibers:
\[
\gamma (\tau,f(X)) = (\gamma \tau, f(\gamma^{-1}(X+\gamma\tau)-\tau)).
\]

We can write the new local coordinate intervening in the action above as a power series as follows: in the fiber above $\gamma \tau$, we have the relation $X = z-\gamma \tau$, and so if $c\neq 0$ we have
\begin{align*}
\gamma^{-1}(X+\gamma \tau)-\tau &= \frac{d(X+\gamma \tau)-b}{-c(X+\gamma\tau)+a}-\tau\\
&=\frac{1}{\det\gamma}\frac{(c\tau+d)^2X}{1-\frac{c(c\tau+d)}{\det\gamma}X}\\
&=\sum_{n\geq 1} c^{n-1}(c\tau+d)^{n+1}\det\gamma^{-n}X^n
\end{align*}
Recalling \ref{ss:notation} the $1$-cocycle $j(\alpha,\tau)$ for $\alpha = \stwomat abcd \in \GL_2^+(\RR)$, one finds that
\begin{equation}
\label{eq:coordinateaction}
\gamma^{-1}(X+\gamma \tau)-\tau =\sum_{n\geq 1} c^{n-1}j(\gamma,\tau)^{n+1}\det\gamma^{-n}X^n.
\end{equation}

\subsubsection{The VOA bundle}
Now consider the fiber product 
\[\cV_{\uhp} = (\uhp \times \Aut_\cO)\times_{\Aut_\cO} V,\] 
where $\Aut_{\cO}$ acts on $V$ as in \cite{FBZ}. Concretely,
\[
\cV_{\uhp} = \left\{(\tau,f(X),v) \in \uhp\times \Aut_{\cO}\times V\right\}/(\tau,f(X)*g(X),v) \sim (\tau,f(X),R(g)v).
\]
In particular, in $\cV_{\uhp}$ we always have 
\[(\tau,f(X),v) = (\tau,X,R(f)v)\]
where $R(f)$ denotes the left action of $f\in \Aut_{\cO}$ on the quasi-vertex operator algebra $V$.\ That means that every fiber of $\cV_{\uhp}$ over $\tau \in \uhp$ can be identified with $V$ using the representatives $(\tau,X,v)$ for $v \in V$. In this way we obtain a trivialization
\[
\cV_{\uhp} \cong \uhp \times V.
\]

\begin{rmk}
    Recall from \cite{FBZ} that all that is needed to define $R(f)$ are the Virasoro operators $L(n)$ for $n\geq 0$ acting on $V$.\ In fact, we will see below that in this simple case of changes of coordinates by linear fractional transformations, only $L(0)$ and $L(1)$ play a r\^{o}le, so that it is really an action on $V$ arising from a subalgebra of the canonical $\sl2$-action on $V$.\ Later, when we describe a connection on $\cV_{Y(1)}$, the operator $L(-1)$ will also play a r\^{o}le, leading to the full action of $\sl2$.\ This circumstance is the main reason that we take $V$ to be a QVOA rather than a VOA.
\end{rmk}

Even though $\cV_{\uhp}$ can be trivialized as above, its more complicated expression as a fiber product is essential for endowing $\cV_{\uhp}$ with an action of $\GL_2^+(\RR)$ that lifts the fractional linear transformation on $\uhp$ . Then, restricting this action to the modular group allows us to descend $\cV_{\uhp}$ to a bundle on $Y(1)$. In this way we arrive at the following definition:

\begin{dfn}\label{dfnHV}
    The bundle $\cV_{Y(1)}$ over the modular curve associated to the QVOA $V$ is defined as the quotient
    \[
    \cV_{Y(1)} \df \Gamma \obackslash (\uhp \times V)
    \]
    under the action
    \[
\gamma(\tau,v) \df (\gamma \tau, R(\gamma^{-1}(X+\gamma\tau)-\tau)v).
\]
\end{dfn}

In order to find descriptions for various operators $R(\rho)$ for $\rho(X) \in \Aut\cO$ that intervene below, we prepare with a basic well-known lemma.
\begin{lem}
\label{l:computeR}
    If $\rho(X) = \sum_{m\geq 1}\alpha^{m-1}X^m$, then
    \[
    \rho(X) = \exp\left(\alpha X^{2}\partial_X\right) X.
    \]
\end{lem}
\begin{proof}
    We verify this directly as follows: first observe by induction that
    \[
    (X^2\partial_X)^mX = m!X^{m+1}.
    \]
    It then follows easily that
    \[
    \exp\left(\alpha X^{2}\partial_X\right) X = \sum_{m\geq 0}\frac{\alpha^m}{m!}(X^2\partial_X)^mX  = \sum_{m\geq 0} \alpha^mX^{m+1} =\rho(X),  
    \]
    as desired.
\end{proof}

Our next goal is to understand the operator $R(\gamma^{-1}(X+\gamma \tau)-\tau)$ on the QVOA $V$.\ We will do this for $\gamma \in \GL_2^+(\RR)$, not just for matrices in the modular group, with a view towards defining Hecke operators below. Given equation \eqref{eq:coordinateaction} above, following \cite{FBZ}, we must solve for scalars $v_i$ so that
\[
\sum_{n\geq 1} c^{n-1}j(\gamma,\tau)^{n+1}\det(\gamma)^{-n}X^n = \exp\left(\sum_{i \geq 1} v_i X^{i+1}\partial_X\right) v_0X.
\]
Lemma \ref{l:computeR} implies that
\[
\sum_{n\geq 1} c^{n-1}j(\gamma,\tau)^{n+1}\det\gamma^{-n}X^n = \exp\left(\det\gamma^{-1}cj(\gamma,\tau)X^{2}\partial_X\right) j(\gamma,\tau)^2 \det\gamma^{-1}X.
\]
This yields the following concrete description of the action of $\GL_2^+(\RR)$ via linear fractional transformations on the bundle $\uhp\times V$:
\begin{equation}
\label{eq:Vaction}
\gamma(\tau,v) = (\gamma \tau, e^{-\det\gamma^{-1}cj(\gamma,\tau)L(1)}j(\gamma,\tau)^{-2L(0)} (\det\gamma)^{L(0)}v).
\end{equation}

\subsubsection{An elementary treatment of (\ref{eq:Vaction}) }
Since we have followed the prescription of \cite{FBZ}, (\ref{eq:Vaction}) necessarily yields a left group action of $\GL_2^+(\RR)$.\ Here we will provide a separate and elementary verification of this fact for readers unfamiliar with the preceding material.

\begin{proof}[Elementary proof that \eqref{eq:Vaction} is a left group action] Let 
\begin{align*}
    \alpha \df\stwomat{a}{b}{c}{d}, \  \beta \df \stwomat{r}{s}{t}{u}
\end{align*} 
be two matrices in $\GL_2^+(\RR)$.\
Recalling the meaning of $X$ in Section \ref{SNotation}, a brute-force calculation shows that
\begin{lem}\label{lemXform}
We have
$$X(\alpha)|_2\beta=X(\alpha\beta)-X(\beta).$$
$\hfill\Box$
\end{lem}

Turning to the proof of (\ref{eq:Vaction}), let $v \in V_k$. Then we have
\begin{align*}
   \beta(\tau,v)&=(\beta\tau, e^{-\det\beta^{-1} tj(\beta, \tau)L(1)}j(\beta,\tau)^{-2L(0)}(\det\beta)^{L(0)}v)\\
   &=(\beta\tau, e^{-\det\beta^{-1}tj(\beta, \tau)L(1)}j(\beta,\tau)^{-2k}\det\beta^{k}v)\\
   &=\left(\beta\tau, \det\beta^kj(\beta, \tau)^{-2k}\sum_{n\geq0} \frac{(-t)^nj(\beta, \tau)^n}{\det\beta^n\cdot n!} L(1)^nv   \right).
\end{align*}
Therefore
\begin{align*}
   & \alpha(\beta(\tau,v))\\
   &=\left(\alpha(\beta\tau), e^{-c\det\alpha^{-1}j(\alpha, \beta\tau)L(1)} j(\alpha, \beta\tau)^{-2L(0)}(\det\alpha)^{L(0)} \det\beta^k j(\beta, \tau)^{-2k}\sum_{n\geq0} \frac{(-t)^nj(\beta, \tau)^n}{\det\beta^n\cdot n!} L(1)^nv \right)\\
   &=\left(\alpha\beta\tau, e^{-c\det\alpha^{-1}j(\alpha, \beta\tau)L(1)} \det(\alpha \beta)^kj(\beta, \tau)^{-2k}j(\alpha, \beta\tau)^{-2k}\sum_{n\geq0} \frac{(-t)^nj(\beta, \tau)^nj(\alpha, \beta\tau)^{2n}}{\det(\alpha \beta)^n\cdot n!} L(1)^nv \right)\\
   &=\left(\alpha\beta\tau, \det(\alpha \beta)^kj(\alpha\beta, \tau)^{-2k}e^{-c\det\alpha^{-1}j(\alpha, \beta\tau)L(1)}\sum_{n\geq0} \frac{(-t)^nj(\alpha\beta, \tau)^nj(\alpha, \beta\tau)^{n}}{\det(\alpha \beta)^n\cdot n!} L(1)^nv \right)\\
   &=\left(\alpha\beta\tau, \det(\alpha \beta)^kj(\alpha\beta, \tau)^{-2k}e^{-(c+tj(\alpha\beta,\tau)\det\beta^{-1})\det\alpha^{-1}j(\alpha, \beta\tau)L(1)}v\right)
\end{align*}
On the other hand
\begin{align*}
  (\alpha\beta)(\tau,v)&= \left(\alpha\beta\tau, e^{-(\det\alpha\beta)^{-1}(cr+dt)j(\alpha \beta, \tau)L(1)} j(\alpha\beta, \tau)^{-2L(0)}(\det\alpha\beta)^{L(0)}v   \right)\\
  &= \left(\alpha\beta\tau,\det(\alpha\beta)^kj(\alpha\beta,\tau)^{-2k} e^{-\det(\alpha\beta)^{-1}(cr+dt)j(\alpha \beta, \tau)L(1)}v   \right)
\end{align*}

Thus, we see that \eqref{eq:Vaction} defines an action as long as the following identity holds for all $\alpha,\beta \in \GL_2^+(\RR)$:
\[
(\det\alpha\beta)^{-1}(cr+dt)j(\alpha\beta,\tau) = (c+tj(\alpha\beta,\tau)\det\beta^{-1}j(\alpha,\beta\tau)\det\alpha^{-1},
\]
and further manipulation shows that this is equivalent to the identity of Lemma
\ref{lemXform}.\ Therefore, this gives an elementary verification that \eqref{eq:Vaction} indeed defines a left action of $\GL_2^+(\RR)$ on $\uhp \times V$.
\end{proof}

\subsection{The one-cocycle \texorpdfstring{$K(\gamma, \tau)$}{K}}
    By definition, sections of the VOA-bundle $\cV_{Y(1)}$ are $V$-valued modular forms that transform under 
    $\Gamma$ according to the action \eqref{eq:Vaction}.\ That is, sections $f$ are functions $f\colon \uhp \to V$ satisfying
    \[
    f(\gamma \tau) = e^{-cj(\gamma,\tau)L(1)}j(\gamma,\tau)^{-2L(0)}f(\tau)
    \]
    for all $\gamma \in \Gamma$.\ Slightly more generally, we are led to introduce the following notation:
    \begin{dfn}\label{dfnK}
        For $\tau \in \uhp$ and $\gamma \in GL_2(\RR)^+$ define the $\Endo(V)$-valued function
        \[
          K(\gamma,\tau) \df e^{-\det\gamma^{-1}cj(\gamma,\tau)L(1)}j(\gamma,\tau)^{-2L(0)} (\det\gamma)^{L(0)}.
        \]
    \end{dfn}

Note that since $L(1)$ and $L(0)$ do not commute, the order of the factors in this expression is important. 
    \begin{lem}\label{lemK1} $K$ is a $1$-cocycle on $\GL_2(\RR)^+$.\ That is, for $\alpha, \beta \in \GL_2(\RR)^+$ we have
    $$ K(\alpha \beta, \tau)=K(\alpha, \beta\tau)K(\beta, \tau)$$
    \end{lem}
    \begin{proof}
       This is equivalent to the statement that (\ref{eq:Vaction}) is a group action. 
    \end{proof}

$K(\gamma, \tau)$ takes a simpler form for certain choices of $\gamma$.\ For example if 
    $\gamma= \stwomat ab0a$ then $K(\gamma, \tau)=\Id_V$.\ We shall be particularly concerned with the case $\gamma\in \Gamma$ in which case, of course, Definition \ref{dfnK} reads
    \[
          K(\gamma,\tau) \df e^{-cj(\gamma,\tau)L(1)}j(\gamma,\tau)^{-2L(0)}\quad\quad (\gamma\in\Gamma).
        \]

Thanks to Lemma \ref{lemK1} we get a right action $f||_k\gamma$ of $\GL_2(\RR)^+$ on sections $f$ as in Section \ref{SNotation}.\ We repeat the definition here because of its centrality to our project: 
\begin{eqnarray}\label{fKact}
    f||_k\gamma(\tau)\df j(\gamma, \tau)^{-k}(\det\gamma)^{k/2}K(\gamma, \tau)^{-1}f(\gamma\tau).
    \end{eqnarray}

 \medskip
 We  call a $\Gamma$-invariant of this action a \textit{weak $V$-valued modular form of weight $k$}, where the adjective "weak" means that we place no analytic conditions on $f(\tau)$.\ 
Here, we just point out, as a consequence of previous comments, that under suitable analytic conditions, such a $V$-valued form $f$ will have a $q$-expansion because it is periodic:
\begin{eqnarray}\label{Tact}
f(\tau)=(f||_k \pm T)(\tau)=f(\tau+1).
\end{eqnarray}
\subsection{Extension to the compactified curve}
\label{ss:cusp}
In this subsection we extend the construction of $\cV_{Y(1)}$ to the compact modular curve. Notice that the action defined in \eqref{eq:Vaction} satisfies $(\pm T)(\tau,v) = (\tau+1,v)$. Therefore, as discussed at the end of Section \ref{ss:cusps}, there is a canonical extension $\cV_{X(1)}$ of $\cV_{Y(1)}$ to the compact modular curve $X(1)$. 

\medskip
Let $\iota_1$ and $\iota_2$ be the covering maps associated to $\cV_{X(1)}$ as in Section \ref{ss:cusps}. By abuse of notation we will write $\iota_1^*\cV_{X(1)}$ and $\iota_2^*\cV_{X(1)}$ for the pullbacks of the parts of $\cV_{X(1)}$ over $Y(1)$ and $C_2\obackslash \bD$, respectively. Given $v \in V$, the function $f(\tau) = (\tau,v)$ is a map $\uhp \to \uhp\times V$, and it is equivariant for the action of $\pm T$:
\[
f(\pm T\tau) = (\tau+1,v) = \pm T(\tau,v) = \pm Tf(\tau).
\]
Therefore, every element $v\in V$ defines a section $(\tau,v)$ of $\iota_1^*\cV_{X(1)}$.

\medskip
Our goal now is to re-express sections of $\cV_{Y(1)}$ in a neighborhood of the cusp, and thereby define $q$-expansions. To do this, let us first work on $\cV_{\uhp}$, which recall is a quotient of $\uhp \times \Aut_{\cO} \times V$. In the coordinates on $\bD^\times$, the point $\tau$ maps to $q=e^{2\pi i \tau}$. Choose another global coordinate $q_z = e^{2\pi i z}$ on $\bD^\times$, so that a local coordinate at $q$ is $q_z -q$. This local coordinate pulls back under the exponential map to a new coordinate
\[
e^{2\pi iz}-e^{2\pi i \tau} = q^{-1}(e^{2\pi iX}-1)
\]
on $\uhp$ at $\tau$, since $X = z-\tau$ on $\uhp$. Therefore, under the change of coordinates provided by the map $q$, we have
\[
(\tau,q^{-1}(e^{2\pi i X}-1),v) \in \cV_{\uhp} \mapsto (q,q_z-q,v) \in \cV_{\bD^\times}.
\]
In $\cV_{\uhp}$ we have the identification
\[
(\tau,q^{-1}(e^{2\pi i X}-1),v) = (\tau,X,R(q^{-1}(e^{2\pi i X}-1))v).
\]
Therefore, near the cusp, the section $(\tau,v)$ of $\cV_{\uhp}$ looks like $(q,R(q^{-1}(e^{2\pi i X}-1))^{-1}v)$ on the punctured $q$-disc $\bD^\times$. Note that one has
\[
R(q^{-1}(e^{2\pi i X}-1))^{-1}v = q^{-L(0)}R(e^{2\pi i X}-1)^{-1}v= q^{-L(0)}R\left(\tfrac 1{2\pi i}\log(X+1)\right)v.
\]

\medskip
Thus, the section $(\tau,v)$ near the cusp of $\uhp$ corresponds to $(q,q^{-L(0)}R(\tfrac 1{2\pi i}\log(X+1))v)$ near $q=0$ in the $q$-disc $\bD^\times$. In particular, in the canonical coordinate invariant extension of $\cV_{Y(1)}$ to the cusp, constant sections $(\tau,v)$ of $\cV_{\uhp}$ on $\uhp$ will acquire a finite order pole at $q=0$ due to the factor of $q^{-L(0)}$. This is similar to the fact that Eisenstein series, in the classical scalar setting of modular forms, technically have poles at the cusp when interpreted geometrically as sections of the canonical bundle on the modular curve.

Thus, just as in the usual modular theory, we define $q$-expansions in usual the naive way. 
\begin{dfn}
\label{d:holomorphicatcusp}
    A holomorphic section $f$ of $\cV_{Y(1)}$ is said to be \emph{holomorphic at the cusp} if it admits a holomorphic $q$-expansion
    \[
      f(q) \in \pseries{V}{q}.
    \]
    Meromorphy at the cusp is defined similarly.
\end{dfn}
\begin{rmk}
Note that since the bundle $\cV_{Y(1)}$ is constructed as an inductive limit of finite rank bundles, the $q$-expansion will lie in some filtered subspace $\pseries{F_NV}{q}$ for some $N\geq 0$.    
\end{rmk}

\subsection{VOA-valued modular forms}
\label{s:VOAmodular}
All of this prompts the following.\ Fix a  quasi-vertex operator algebra $V$. 
\begin{dfn} Suppose that for some $N$,
$$f:\mathcal{H}\rightarrow F_NV$$
is a holomorphic map.\ We say that $f$ is a \emph{weakly holomorphic $V$-valued modular form of weight $k$} (and conformal degree at most $N$) if it satisfies
$$f||_k\gamma=f\quad\quad (\gamma\in\Gamma).$$
\end{dfn}

\begin{rmk}
  Because of the nature of the grading on $V$ 
  (cf.\ Appendix B.0.5) $F_NV$ is finite-dimensional so that it is sensible to require that $f$ is holomorphic.
\end{rmk}

\begin{dfn} A weakly holomorphic $V$-valued modular form $f$ is a \emph{meromorphic}, respectively \emph{holomorphic}, if it is meromorphic or holomorphic at the cusp in the sense of Definition \ref{d:holomorphicatcusp}.
\end{dfn}

We now introduce spaces of $V$-valued modular forms. Let $M_k(V)$ denote the set of weight-$k$ $V$-valued holomorphic forms, and set
\begin{align*}
M(V) &\df \bigoplus_{k\in\ZZ} M_k(V).
\end{align*}
Each $M_k(V)$ has an increasing conformal filtration whose conformal degree $N$ piece $F_NM_k(V)$, consisting of those forms taking values in $F_NV$. This endows $M(V)$ with an increasing conformal filtration
\[
F_NM(V) = \bigoplus_{k\in\ZZ} F_NM_k(V).    
\]
Since every $V$-valued modular form is assumed to be of finite conformal degree, the conformal filtration on $M(V)$ is exhaustive:
\[
M(V) = \bigcup_{N \geq 0} F_NM(V).    
\]

For ease of reference, let us summarize how elements of $M_k(V)$ are defined. An element of $M_k(V)$ consists of a holomorphic function of bounded conformal degree
\[
f\colon \uhp \to F_NV 
\]
satisfying the following properties:
\begin{enumerate}
    \item modularity: $f(\gamma \tau) = j(\gamma,\tau)^kK(\gamma,\tau)f(\tau)$, i.e., $f||_k\gamma=f $ for all $\gamma \in \Gamma$;
    \item holomorphy at the cusp: $f$ can be expanded as a holomorphic $q$-series $f(q) \in \pseries{F_NV}{q}$.
\end{enumerate}

\section{The connection}
\label{s:connection}

The bundle $\cV_{Y(1)}$ is endowed with a connection
    \[
    \nabla \colon \cV_{Y(1)} \to \cV_{Y(1)}\otimes \Omega^1_{Y(1)}
    \]
    defined locally on sections by the formula $\nabla = d + L(-1)\otimes d\tau$. The text \cite{FBZ} defers the proof that this is a connection until quite late in the exposition, after the ground has been prepared to give a simpler conceptual argument. Therefore, we provide an elementary proof of this fact in Section \ref{ss:nablaproof} below for the convenience of readers new to this material.

\subsection{Proof that \texorpdfstring{$\nabla$}{it} is a connection}
\label{ss:nablaproof}
To verify this, let us first determine what transformation law is satisfied by $df$. In order to lay the groundwork for defining an analog of the modular derivative in higher weights, we allow $f$ to be a $V$-valued modular form of possibly nonzero weight $k$. Therefore we assume that $f$ satisfies the transformation law 
\begin{equation}
\label{eq:modwtk}
    f(\gamma \tau) = e^{-cj(\gamma,\tau)L(1)}j(\gamma,\tau)^{k-2L(0)}f(\tau).
\end{equation}
Differentiating the left hand side of equation \eqref{eq:modwtk} yields:
\[
d(f(\gamma\tau)) = j(\gamma,\tau)^{-2}df(\gamma\tau).
\]
On the other hand, differentiating the right hand side of equation \eqref{eq:modwtk} gives
\begin{align*}
&-c^2L(1)e^{-cj(\gamma,\tau)L(1)}j(\gamma,\tau)^{-2L(0)}f(\tau)+ce^{-cj(\gamma,\tau)L(1)}(k-2L(0))j(\gamma,\tau)^{k-2L(0)-1}f(\tau)+\\
&e^{-cj(\gamma,\tau)L(1)}j(\gamma,\tau)^{-2L(0)}df(\tau)\\
=&e^{-cj(\gamma,\tau)L(1)}(-c^2L(1)j(\gamma,\tau)^{k-2L(0)}f(\tau)+c(k-2L(0))j(\gamma,\tau)^{k-2L(0)-1}f(\tau)+j(\gamma,\tau)^{k-2L(0)}df(\tau))\\
=&\left(\frac{c(k-2L(0))}{j(\gamma,\tau)}-c^2L(1)\right)f(\gamma\tau)+e^{-cj(\gamma,\tau)L(1)}j(\gamma,\tau)^{k-2L(0)}df(\tau)
\end{align*}
Thus, we find that $df$ satisfies the transformation law
\begin{align*}
    df(\gamma\tau) &= \left(cj(\gamma,\tau)(k-2L(0))-(cj(\gamma,\tau))^2L(1)\right)f(\gamma\tau)+e^{-cj(\gamma,\tau)L(1)}j(\gamma,\tau)^{k+2-2L(0)}df(\tau)
\end{align*}
Therefore,
\begin{align*}
    (\nabla f)(\gamma \tau) =& \left(cj(\gamma,\tau)(k-2L(0))-(cj(\gamma,\tau))^2L(1)+L(-1)\right)f(\gamma\tau)+\\
    &\quad e^{-cj(\gamma,\tau)L(1)}j(\gamma,\tau)^{k+2-2L(0)}df(\tau)
\end{align*}

Recall that $E_2(\tau)$ 
satisfies the transformation law
\begin{equation}
\label{eq:E2modulartransformation}
E_2(\gamma\tau) = j(\gamma,\tau)^2E_2(\tau) +\frac{12 cj(\gamma,\tau)}{2\pi i}.    
\end{equation}
\begin{dfn}
    \label{d:moderiv}
    The $k$th modular derivative with respect to the QVOA $V$ is the operator $\nabla_{k}$ that acts on functions $f\colon \uhp \to V$ via the formula
    \[
    \nabla_{k} f = \nabla f -\frac{2\pi i k}{12}E_2f.
    \]
\end{dfn}
If $f$ satisfies the transformation law of equation \eqref{eq:modwtk}, then we find from the computations above, and the transformation law for $E_2$, that
\begin{align*}
    (\nabla_kf)(\gamma\tau) &= \left(L(-1)-2cj(\gamma,\tau)L(0)-(cj(\gamma,\tau))^2L(1)\right)f(\gamma\tau)+\\
    &\quad e^{-cj(\gamma,\tau)L(1)}j(\gamma,\tau)^{k+2-2L(0)}df(\tau)-\frac{2\pi i k}{12}j(\gamma,\tau)^2E_2(\tau)f(\gamma \tau)
\end{align*}

We want to show that this transforms according to the law
\begin{equation}
\label{eq:modderiv}
(\nabla_k f)(\gamma \tau) = e^{-cj(\gamma,\tau)L(1)}j(\gamma,\tau)^{k+2-2L(0)}(\nabla_kf)(\tau).    
\end{equation}
Our computations above have shown that by adding the factor of $E_2$ in the definition of $\nabla_k$, this essentially reduces the proof of \eqref{eq:modderiv} to weight zero, in which case equation \eqref{eq:modderiv} simply amounts to the statement that $\nabla$ is indeed a connection. Thus, equation \eqref{eq:modderiv} follows by the general theory outlined in \cite{FBZ}.

\medskip
Here is a direct proof of \eqref{eq:modderiv}: by the computations above, establishing \eqref{eq:modderiv} is tantamount to establishing the following identity:
\begin{align*}
L(-1)=&j(\gamma,\tau)^{2L(0)-2}e^{cj(\gamma,\tau)L(1)}L(-1)e^{-cj(\gamma,\tau)L(1)}j(\gamma,\tau)^{-2L(0)}\\
 &-c^2j(\gamma,\tau)^{2L(0)-2}L(1)j(\gamma,\tau)^{2-2L(0)}-2cj(\gamma,\tau)^{-1}L(0)
\end{align*}
We will verify this one graded piece at a time. 

\medskip
Let $f_N \in V_N$, so that by linearity, to establish equation \eqref{eq:modderiv}, it suffices to establish the following identity:
\begin{align}
\label{eq:connectionproof}
 L(-1)f_N=&j(\gamma,\tau)^{2L(0)-2N-2}e^{cj(\gamma,\tau)L(1)}L(-1)e^{-cj(\gamma,\tau)L(1)}f_N-c^2j(\gamma,\tau)^{-2}L(1)f_N\\
\nonumber &-2Ncj(\gamma,\tau)^{-1}f_N   
\end{align}

The left hand side of equation \eqref{eq:connectionproof} is equal to $L(-1)f_N$ in degree $N+1$, and it is zero in all other degrees. The right hand side of equation \eqref{eq:connectionproof} vanishes in degrees greater than $N+1$. The degree $N+1$ part of the right hand side of equation \eqref{eq:connectionproof} is
\[
j(\gamma,\tau)^{2L(0)-2N-2}L(-1)f_N,
\]
and this equals $L(-1)f_N$ since this is homogeneous of degree $N+1$. Therefore, equation \eqref{eq:connectionproof} is true in homogeneous degrees $\geq N+1$.

\medskip
It remains to prove that the degree $n$ terms of the right hand side of equation \eqref{eq:connectionproof} vanish when $n \leq N$.\
The degree $N$ term of the right hand side of \eqref{eq:connectionproof} is equal to
\begin{align*}
&j(\gamma,\tau)^{2L(0)-2N-2}cj(\gamma,\tau)(L(1)L(-1)-L(-1)L(1))f_N-2Ncj(\gamma,\tau)^{-1}f_N\\
=&j(\gamma,\tau)^{-1}c((L(1)L(-1)-L(-1)L(1))-2N)f_N
\end{align*}
This vanishes thanks to the commutator relation $[L(1),L(-1)]=  2L(0)$ in the Virasoro algebra, since $f_N$ is homogeneous of degree $N$. Therefore, this confirms our identity in degree $N$.

\medskip
Now compute the degree $N-1$ term of the right hand side of \eqref{eq:connectionproof}. This time we obtain
\[
c^2j(\gamma,\tau)^{-2}(\tfrac 12 L(-1)L(1)^2+\tfrac 12 L(1)^2L(-1)-L(1)L(-1)L(1))f_N-c^2j(\gamma,\tau)^{-2}L(1)f_N
\]
The Virasoro relations imply that
\[\tfrac 12 L(-1)L(1)^2+\tfrac 12 L(1)^2L(-1)-L(1)L(-1)L(1) = L_1,\] 
so that this does cancel to zero, and this confirms equation \eqref{eq:connectionproof} in degree $N-1$.

\medskip
Finally, in degree $N-n < n-1$ (so $2 \leq n \leq N$), the desired vanishing follows from the vanishing of the following operator in the Virasoro algebra:
\[
\sum_{j=0}^n (-1)^{n-j}\frac{1}{j!(n-j)!}L(1)^jL(-1)L(1)^{n-j} = \frac{(-1)^n}{n!}\sum_{j=0}^n(-1)^{j}\binom{n}{j} L(1)^jL(-1)L(1)^{n-j}.
\]
Again, we provide details here for completeness.\ When $n=2$ we must show that the following expression vanishes:
\begin{align*}
&L(-1)L(1)^2 - 2L(1)L(-1)L(1)+L(1)^2L(-1)\\
=& (L(1)L(-1)-2L(0))L(1) - 2L(1)L(-1)L(1)+L(1)(L(-1)L(1)+2L(0))\\
=&0.
\end{align*}
So our vanishing result holds when $n=2$. In general for $n\geq 3$ we write:
\begin{align*}
&\sum_{j=0}^n(-1)^{j}\binom{n}{j} L(1)^jL(-1)L(1)^{n-j}\\
=& (-1)^nL(1)^nL(-1)+L(-1)L(1)^n+\sum_{j=1}^{n-1}(-1)^{j}\binom{n}{j} L(1)^jL(-1)L(1)^{n-j}\\
=&(-1)^nL(1)^nL(-1)+L(-1)L(1)^n+\sum_{j=1}^{n-1}(-1)^{j}\left(\binom{n-1}{j-1}+\binom{n-1}{j}\right) L(1)^jL(-1)L(1)^{n-j}\\
=&(-1)^nL(1)^nL(-1)+L(-1)L(1)^n+\sum_{j=1}^{n-1}(-1)^{j}\binom{n-1}{j-1} L(1)^jL(-1)L(1)^{n-j}+\\
&\sum_{j=1}^{n-1}(-1)^{j}\binom{n-1}{j}L(1)^jL(-1)L(1)^{n-j}\\
=&-L(1)\left(\sum_{j=1}^{n}(-1)^{j}\binom{n-1}{j-1} L(1)^{j-1}L(-1)L(1)^{n-j}\right)+\\
&\left(\sum_{j=0}^{n-1}(-1)^{j}\binom{n-1}{j}L(1)^jL(-1)L(1)^{n-1-j}\right)L(1)
\end{align*}
Each of the sums in the last line above vanish by induction, and this concludes the proof of equation \eqref{eq:connectionproof}. This yields both equation \eqref{eq:modderiv}, and that $\nabla = d + L(-1)\otimes d\tau $ defines a connection on $\cV_{Y(1)}$.

\section{QVOA-valued quasi-modular forms}
\label{s:qmf}
\subsection{Some motivation}
\label{ss:qmfmotivation}
Fix a QVOA $V=(V, Y, \mathbf{1}, \rho)$.\ In this Subsection we consider some rudimentary aspects of the space $M(V)$ of $V$-valued modular forms.\ These simple ideas inform later Sections and presage some of the main results.\
Recall the meaning of $X=X(\gamma, \tau)$ from
Subsection \ref{ss:notation}.

\begin{ex}
\label{ex:constantsection}
Begin with a state $v\in V_m$ of conformal degree $m$.\ We first ask if the constant map
$$f:\uhp\rightarrow V_m,\quad\quad \tau\mapsto v$$
belongs to $M(V)$.\ Certainly $f$ is  holomorphic.\ For $\gamma\in\Gamma$ we compute
\begin{eqnarray}\label{cfcalc}
 &&f||_{2m}\gamma(\tau) = j(\gamma, \tau)^{-2m}K(\gamma, \tau)^{-1}f(\tau) \notag  \\
 &&=j(\gamma, \tau)^{2L(0)-2m}e^{cj(\gamma, \tau)L(1)}v  \notag \\
 &&=\sum_{n\geq0} \frac{c^nj(\gamma, \tau)^{n-2m+2(m-n)}}{n!} L(1)^nv  \notag \\
 &&=\sum_{n\geq0} \frac{1}{n!} X^n L(1)^nv.
\end{eqnarray}
For the last equality we used $L(1)^nv\in V_{m-n}$.\
From this we see that $f\in M_{2m}(V)$
if, and only if, $L(1)v=0$, i.e., $v$ is a \textit{quasi-primary state}.\ Note that (\ref{cfcalc}) is redolent of the definition of a quasi-modular form of weight $2m$.\ Indeed, setting $Q_n(f):=\frac{1}{n!} L(1)^nf$ yields
(at least formally) the definition of quasi-modular form, albeit one taking values in $V$ rather than $\CC$.
\end{ex}

Before turning to our second
 illustrative example, consider an arbitrary pair
 of elements $f, g\in M(V)$.\ For an integer $n$ we can define the $n^{th}$ product $f(n)g$ in a pointwise manner: 
 \begin{equation}
 \label{fngdef}
 (f(n)g)(\tau)\df f(\tau)(n)g(\tau).    
 \end{equation}
 
In this way we have equipped $M(V)$
with $n^{th}$ products for all $n$ so it is natural to ask if $M(V)$ is itself some sort of vertex algebra.\ The answer is ``no'' for the simple reason that $M(V)$ is \emph{not closed} with respect to these products.\ 

\begin{ex}
Consider a pair of constant maps
$$f:\uhp\rightarrow  u, \quad\quad g:\uhp\rightarrow v$$
and assume that $u$ and $v$ are both quasiprimary states in $V$, whereas $u(n)v$ is \emph{not}. This is a commonplace occurrence: for example in the rank $1$ Heisenberg VOA the canonical weight $1$ state $h$ is quasi-primary whereas $h(-2)h$ is not. Then we just saw in the previous example that $f$ and $g$ both belong to $M(V)$.\ On the other hand since $u(n)v$ is assumed to be not quasi-primary then the same example shows that $f(n)g$ does \emph{not} belong to $M(V)$. Thus, in general $M(V)$ is not closed under the natural pointwise modes of equation \eqref{fngdef}.
\end{ex}

These examples motivate us to introduce quasi-modular forms.
 
\subsection{Definition and first properties of \texorpdfstring{$V$}{V}-valued quasi-modular forms}
\label{ss:qmfdefinition}

Apropos of the preceding discussion we begin with the following basic definition.
\begin{dfn}
\label{d:QMV}
 A $V$-valued meromorphic \emph{quasi-modular form} of weight $k$ and depth at most $s$ is a holomorphic map
 \[
 f\colon \uhp \to F_NV
 \]
 for some $N$, such that 
 \[
   f||_k\gamma(\tau)=\sum_{n=0}^s X^nQ_n(f)
 \]
 for some holomorphic maps $Q_n(f) \colon \uhp \to F_NV$, and such that $f$ is meromorphic at the cusp.\ As usual, we say that $f$ is holomorphic if its $q$-expansion at the cusp has no pole.
\end{dfn}

We make some standard observations in this set-up.\ By considering the cases $\gamma=\pm T$ (when $K(\gamma, \tau)=\Id$) we find that $f(\tau+1)=f(\tau)$ so that $f$ has a $q$-expansion and the Definition has meaning.\ Furthermore $Q_0(f)=f$ and
$f=0$ if $k$ is odd.

\begin{ex}
    Let $v \in V_m$ and let $f:\uhp\rightarrow v$ be the constant function as in the Example \ref{ex:constantsection}. Then our computations  \eqref{cfcalc} show that 
    $$f||_{2m}\gamma(\tau)
    =e^{XL(1)}v.$$ 
    Thus, $f$ is a holomorphic $V$-valued quasi-modular form of weight $2m$ and its depth is the largest integer $s$ such that $L(1)^sv\neq0$. Note that $s$ exists because $L(1)$ is a lowering operator and $V_n=0$ for $n\ll0$.
\end{ex}

Let $Q_k(V)$ be the space of holomorphic $V$-valued quasi-modular forms of weight $k$ and set $Q(V) = \bigoplus_{k\geq 0}Q_{2k}(V)$. Then $Q(V)$ possesses an increasing filtration by depth.
\begin{thm}
    \label{p:QMVtensor} Let $V$ be any QVOA and regard $Q\otimes V^{(2)}$ as a $2\ZZ$-graded linear space with the tensor product grading.\ Then there is a canonical  identification of $2\ZZ$-graded $Q$-modules
    \begin{eqnarray*}
     \iota: Q\otimes V^{(2)}&\stackrel{\cong}{\longrightarrow} Q(V),\\
     f(\tau)\otimes v &\mapsto  f(\tau)v.
    \end{eqnarray*}
    Here, $V^{(2)}$ is the doubled VOA, cf.\ Subsection \ref{SNotation} and Appendix \ref{AppB}.
\end{thm}
\begin{proof}
Choose any $f(\tau)\in Q_{2k}$ and $v\in V_{2m}^{(2)}=V_{m}$.\ First we claim that
$f(\tau)v\in Q_{2(k+m)}(V)$.\ The proof is an elaboration of Example \ref{ex:constantsection}.\ Thus there are holomorphic functions $Q_n(f)$ such that
$$f|_{2k}\gamma(\tau)=\sum_{n=0}^s X^nQ_n(f),$$
and we compute
\begin{eqnarray*}
&&(f(\tau)v)||_{2k+2m}\gamma(\tau) \\
&&=    j(\gamma, \tau)^{-2k-2m+2L(0)}e^{cj(\gamma, \tau) L(1)}f(\gamma\tau)v\\
&&=j(\gamma, \tau)^{-2k-2m+2L(0)}\sum_{r\geq0} \frac{c^rj(\gamma, \tau)^r}{r!}L(1)^r f(\gamma\tau)v\\
&&=j(\gamma, \tau)^{-2k-2r}\sum_{r\geq0} \frac{c^rj(\gamma, \tau)^r}{r!}L(1)^r f(\gamma\tau)v\\
&&=j(\gamma, \tau)^{-2k}\sum_{r\geq0} \frac{X^r}{r!}L(1)^r f(\gamma\tau)v\\
&&=\sum_{r\geq0} \frac{X^r}{r!}L(1)^r f|_{2k}(\tau)v\\
&&=\sum_{r\geq0}\sum_{n=0}^s \frac{X^{r+n}}{r!}Q_n(f)L(1)^r v,
\end{eqnarray*}
and this proves the claim. The map $\iota$ is clearly injective, and our calculation shows that it is a graded map of $Q$-modules.\ Thus it remains to establish surjectivity.

\medskip
With this in mind, fix $g\in Q_{2k}(V)$.\ Thus there are holomorphic functions
$Q_d(g):\uhp\rightarrow F_NV$ for some $N$ satisfying
$$g||_{2k}\gamma(\tau)=\sum_{d=0}^s X^dQ_d(g).$$

Let $(v_{ij})_{j=1}^{n_i}$ be a basis for $V_i$ for each $i$, so that $g = \sum_{i,j}g_{ij}v_{ij}$ for 
scalar-valued functions $g_{ij}$, and write
\[
Q_d(g) = \sum_{i\leq N}\sum_{j=1}^{n_i}Q_d(g)_{ij}v_{ij}.
\]
for scalar-valued functions $Q_d(g)_{ij}$.\ We must show that each $g_{ij}$ is quasi-modular. Hence, let us compute:
\begin{align*}
    g(\gamma\tau) &= \sum_{d=0}^sX^d\sum_{i=0}^N\sum_{j=1}^{n_i} Q_d(g)_{ij}(\tau) j(\gamma,\tau)^ke^{-cj(\gamma,\tau)L(1)}j(\gamma,\tau)^{-2L(0)}v_{ij}\\
    &=\sum_{d=0}^sX^d\sum_{i=0}^Nj(\gamma,\tau)^{k-2i}\sum_{j=1}^{n_i} Q_d(g)_{ij}(\tau) e^{-cj(\gamma,\tau)L(1)}v_{ij}\\
    &\equiv \sum_{d=0}^sX^dj(\gamma,\tau)^{k-2N}\sum_{j=1}^{n_N} Q_d(g)_{Nj}(\tau) v_{Nj}\pmod{F_{N-1}V}
\end{align*}
It follows that for each $j$ we have
\[
g_{Nj}|_{2N-k}\gamma (\tau) = \sum_{d=0}^sX^dQ_d(g)_{Nj}(\tau).
\]
That is, the highest conformal degree coordinates $g_{Nj}$ of $g$ are quasi-modular forms of weight $2N-k$.\ Hence if we set $g = \sum_{j=1}^{N_n} f_{Nj}v_{Nj}$, we then have $g \in Q\otimes V$, and $f-g \in Q(V)$ has conformal degree $\leq N-1$.\ Now we may proceed inductively to complete the proof of the Theorem.
\end{proof}

\subsection{Modes of modular forms}
\label{ss:modes}
Suppose that $f$ and $g$ are two $V$-valued quasimodular forms for some VOA $V$, say of weights $k$ and $\ell$ respectively. It follows by Proposition \ref{p:QMVtensor} that $f$ and $g$ are both elements of $Q\otimes V^{(2)}$, and hence the pointwise mode
\[
(f(m)g)(\tau) \df f(\tau)(m)g(\tau)
\]
gives a well-defined element of $Q\otimes V^{(2)}$, of weight $k+\ell-2m-2$. We have already explained that if one assumes further that $f$ and $g$ are both modular, it is not in general true that $f(m)g$ is modular, only quasi-modular.

When $m\geq -1$, one can use a transformation formula of Huang \cite{Huang, FBZ} to add more precision to the quasi-modular transformation properties of $f(m)g$. We first consider a two-variable mode
\[
f(\tau_1)(m)g(\tau_2)
\]
for two independent and distinct points $\tau_1,\tau_2 \in \uhp$. We restrict to modes $m\geq -1$ so that we can give a uniform proof of the next result without cases.
\begin{prop}
\label{p:positivemodes}
    If $f$ and $g$ are VOA-valued modular forms, of weights $k$ and $\ell$ respectively, then for $m\geq -1$, the $m$th mode $f(\tau_1)(m)g(\tau_2)$ satisfies the transformation law
 \begin{align}
        \label{eq:prodtran}
        &f(\gamma\tau_1)(m)g(\gamma\tau_2)\\
        \nonumber =&j(\gamma,\tau_1)^kj(\gamma,\tau_2)^{\ell-2m-2}K(\gamma,\tau_2)\left(\sum_{n=0}^{\infty}\binom{n+m+1}{m+1}(-X(\gamma, \tau_2))^2f(\tau_1)(n+m)g(\tau_2)\right).
    \end{align}
    where $\tau_1,\tau_2 \in \uhp$ are distinct points and $\gamma \in \Gamma$.
\end{prop}
\begin{proof}
    First let $A\in V$ be any state. If we view $\tau_2$ as being fixed and $\tau_1$ as varying, then we have a local coordinate $\tau_1-\tau_2$ at $\tau_1=\tau_2$. For $\gamma \in \Gamma$, we then have two coordinates at $\tau_1 = \gamma \tau_2$: one is $\tau_1 - \gamma \tau_2$, and the other is $\gamma^{-1}\tau_1 - \tau_2$.  These coordinates differ by a change of variable, and we shall prove our theorem by comparing Huang's transformation formula with respect to these two local coordinates. Recall that his formula says the following: if $z = \tau_1-\gamma \tau_2$ then 
    \[
     Y(A,z) = R(\rho)Y(R(\rho_z)^{-1}A,\rho(z))R(\rho)^{-1}
    \]
    where $\gamma^{-1}\tau_1-\tau_2 = \rho(z)$. Recall that we saw earlier that
    \[
    \gamma^{-1}\tau_1-\tau_2 = \sum_{n\geq 1} c^{n-1}j(\gamma,\tau_2)^{n+1}(\tau_1-\gamma\tau_2)^n.
    \]
   and hence
    \[
    R(\rho) = e^{-cj(\gamma,\tau_2)L(1)}j(\gamma,\tau_2)^{-2L(0)}.
    \]
    We also need to determine $R(\rho_z)$ where $z = \tau_1-\gamma\tau_2$. Since we can write 
    \[\rho(X) = \gamma^{-1}(X+\gamma \tau_2)-\tau_2,\] 
    we find that
    \begin{align*}
        \rho_z(X) &=\rho(X+z)-\rho(z)\\
        &= \gamma^{-1}(X+z+\gamma \tau_2)-\tau_2 - (\gamma^{-1}(z+\gamma\tau_2) - \tau_2)\\
        &=\gamma^{-1}(X+\tau_1) - \gamma^{-1}\tau_1
    \end{align*}
    Therefore, by a previous computation setting $\tau = \gamma^{-1}\tau_1$, we have
    \[
        \rho_z(X) = \sum_{n\geq 1} c^{n-1}j(\gamma, \gamma^{-1}\tau_1)^{n+1}X^n.
    \]
    It follows that
    \[
      R(\rho_z) = e^{-cj(\gamma,\gamma^{-1}\tau_1)}j(\gamma,\gamma^{-1}\tau_1)^{-2L(0)}.
    \]
    Thus, in terms of our notation $K(\gamma,\tau) = e^{-cj(\gamma,\tau)}j(\gamma,\tau)^{-2L(0)}$, we have
    \begin{align*}
    R(\rho) &= K(\gamma, \tau_2),\\
    R(\rho_z) &= K(\gamma,\gamma^{-1}\tau_1).
    \end{align*}
    
    Putting all this together, Huang's formula is the identity:
    \begin{equation}
        \label{eq:huang1}
        Y(A,\tau_1-\gamma \tau_2) = K(\gamma, \tau_2)Y(K(\gamma,\gamma^{-1}\tau_1)^{-1}A,\gamma^{-1}\tau_1-\tau_2) K(\gamma,\tau_2)^{-1}.
    \end{equation}
    This holds for $\tau_1\neq \gamma \tau_2$.

    Recall that in terms of our new notation, VOA valued modular forms satisfy the transformation law
    \[
      f(\gamma \tau) = j(\gamma,\tau)^kK(\gamma,\tau)f(\tau).
    \]
    In particular, taking $\tau = \gamma^{-1}\tau_1$, this says
    \[
      f(\tau_1) = j(\gamma,\gamma^{-1}\tau_1)^kK(\gamma,\gamma^{-1}\tau_1)f(\gamma^{-1}\tau_1).
    \]
    Thus, if we set $A = f(\tau_1)$ in Huang's formula \eqref{eq:huang1}, where $f$ is a VOA-valued modular form, then we deduce that
        \begin{equation}
        \label{eq:huang2}
        Y(f(\tau_1),\tau_1-\gamma \tau_2) = j(\gamma,\gamma^{-1}\tau_1)^kK(\gamma, \tau_2)Y(f(\gamma^{-1}\tau_1),\gamma^{-1}\tau_1-\tau_2) K(\gamma,\tau_2)^{-1}.
    \end{equation}
    Now, if $g(\tau_2)$ is another VOA-valued modular form, we can multiply on the right by $g(\gamma \tau_2)$ to deduce that
        \begin{equation}
        \label{eq:huang3}
        Y(f(\tau_1),\tau_1-\gamma \tau_2)g(\gamma \tau_2) = j(\gamma, \gamma^{-1}\tau_1)^kj(\gamma,\tau_2)^{\ell}K(\gamma, \tau_2)Y(f(\gamma^{-1}\tau_1),\gamma^{-1}\tau_1-\tau_2) g(\tau_2).
    \end{equation}

    We now complete the proof by comparing the principal parts of both sides of equation \eqref{eq:huang3}. The extra weight-increasing automorphy factors arise by expressing $\gamma^{-1}\tau_1-\tau_2$ on the right hand side of equation \eqref{eq:huang3} in terms of $\tau_1-\gamma\tau_2$.

    Indeed, to simplify notation, if we set $B = f(\gamma^{-1}\tau_1)$, $E = \Endo(V)$, and $T=\tau_1-\gamma\tau_2$, then we find that
    \begin{align*}
    &Y(B,\gamma^{-1}\tau_1-\tau_2)\\
    \equiv& \sum_{n\geq -1} B(n) (\gamma^{-1}\tau_1-\tau_2)^{-n-1}\pmod{T\pseries{E}{T}}\\
    \equiv& \sum_{n\geq -1} B(n) \left(\sum_{m\geq 1} c^{m-1}j(\gamma,\tau_2)^{m+1}T^m\right)^{-n-1}\pmod{T\pseries{E}{T}}\\
    \equiv& \sum_{n\geq -1} j(\gamma,\tau_2)^{-2n-2}B(n)T^{-n-1} \left(\sum_{m\geq 1} \left(cj(\gamma,\tau_2)T\right)^{m-1}\right)^{-n-1}\pmod{T\pseries{E}{T}}\\
    \equiv& \sum_{n\geq -1} j(\gamma,\tau_2)^{-2n-2}B(n)T^{-n-1} \left(1-cj(\gamma,\tau_2)T\right)^{n+1}\pmod{T\pseries{E}{T}}\\
    \equiv& \sum_{n\geq -1}\sum_{m=0}^{n+1}(-1)^{n+1-m}\binom{n+1}{m}c^{n+1-m} j(\gamma,\tau_2)^{-m-n-1}B(n)T^{-m} \pmod{T\pseries{E}{T}}\\
    \equiv& \sum_{m=0}^\infty(-c)^{-m}j(\gamma,\tau_2)^{-m}\left(\sum_{n=m-1}^\infty\binom{n+1}{m}(-c)^{n+1} j(\gamma,\tau_2)^{-n-1}B(n)\right)T^{-m} \pmod{T\pseries{E}{T}}\\
    \equiv& \sum_{m=-1}^\infty j(\gamma,\tau_2)^{-2m-2}\left(\sum_{n=0}^\infty\binom{n+m+1}{m+1}(-c)^{n} j(\gamma,\tau_2)^{-n}B(n+m)\right)T^{-m-1} \pmod{T\pseries{E}{T}}
    \end{align*}
    Hence, combining this with equation \eqref{eq:huang3} we obtain for $m\geq -1$:
    \begin{align}
        \label{eq:prodtransformation}
        &f(\tau_1)(m)g(\gamma\tau_2)\\
        \nonumber =&j(\gamma,\gamma^{-1}\tau_1)^kj(\gamma,\tau_2)^{\ell-2m-2}K(\gamma,\tau_2)\left(\sum_{n=0}^{\infty}\binom{n+m+1}{m+1}(-c)^nj(\gamma,\tau_2)^{-n}f(\gamma^{-1}\tau_1)(n+m)g(\tau_2)\right).
    \end{align}
    Finally, if we replace $\tau_1$ by $\gamma \tau_1$, this proves the proposition.
    \end{proof}

Our next goal is to explain why, for $m\geq -1$, the transformation law above extends to the diagonal $\tau_1=\tau_2$.
\begin{thm}
    \label{t:modes}
    If $f$ and $g$ are $V$-valued modular forms of weights $k$ and $\ell$, respectively, then for $m\geq -1$, the $m$th mode $f(m)g$ defined by
    $(f(m)g)(\tau) = f(\tau)(m)g(\tau)$
    is analytic on $\uhp$ and satisfies
    \begin{align}
        \label{eq:prodtran2}
        &(f(m)g)(\gamma\tau)\\
        \nonumber =&j(\gamma,\tau)^{k+\ell-2m-2}K(\gamma,\tau)\left(\sum_{n=0}^{\infty}\binom{m+n+1}{m+1}(-X)^nf(\tau)(m+n)g(\tau)\right).
    \end{align}
    for all $\tau \in \uhp$ and $\gamma \in \Gamma$. 
    
    In particular, if $f$ and $g$ have conformal degrees $M$ and $N$, respectively, then $f(m)g$ satisfies the transformation law of a $V$-valued quasi-modular function of weight $k+\ell-2m-2$, depth at most $M+N-m-1$, and conformal degree at most $M+N-m-1$
\end{thm}
\begin{proof}
    Since $f$ and $g$ are of finite conformal degree $N$, say, we can regard the $m$th mode intervening in their product as a bilinear map
    \[
    (m) \colon F_NV \times F_NV \to F_{N-m}V.
    \]
    This is a bilinear map between finite dimensional complex vector spaces, and so it is analytic for the canonical complex structure on these spaces. Since $f$ and $g$ are analytic functions $\uhp \to F_NV$, their product likewise defines an analytic map 
    \[f\times g \colon \uhp\times\uhp \to F_NV^2,\] 
    and the composed map:
    \[
    f(m)g\colon \uhp \times \uhp \to F_{N-m}V
    \]
    is therefore complex analytic on all of $\uhp\times \uhp$. The transformation law of Proposition \ref{p:positivemodes}, which was only proved there for $\tau_1\neq \tau_2$, then extends to the diagonal of $\uhp^2$ by analytic continuation. This proves the first part of the result.

    For the statements about depth and conformal degree, observe that the highest degree piece of $f(\tau)(n+m)g(\tau)$ lives in conformal degree at most $M+N-m-n-1$. When $n=0$ this yields the bound on the conformal degree of $f(m)g$. Then, since $V_{s} = 0$ for $s < 0$, it follows that $f(m+n)g=0$ for $n \geq M+N-m$. Hence $f(m+n)g$ has depth at most $M+N-m-1$.
\end{proof}

\begin{rmk}
We emphasize that modes $f(m)g$ are defined for \emph{all} $m \in \ZZ$ if $f,g \in Q(V)$, by Proposition \ref{p:QMVtensor}, and the import of Theorem \ref{t:modes} is the equation \eqref{eq:prodtran2} explaining how the modes $f(m+n)g$ intervene in the quasi-modular transformation law of $f(m)g$ when $m\geq -1$. We do not require Theorem \ref{t:modes} in what follows, which is why we have not given a definitive statement for all integers $m$.
\end{rmk}

\section{\texorpdfstring{$Q(V)$}{Q(V)} qua QVOA}
\label{s:QVOA}
The purpose of this Section is to describe the algebraic structure of the space $Q(V)$ of quasi-modular forms with values in a QVOA $V$.\ The reader is referred to Appendix \ref{AppB} for further details concerning QVOAs and related structures.

We have already considered $n^{th}$ pointwise products of $V$-valued modular forms \eqref{fngdef} and $V$-valued quasi-modular forms.\ In this way $Q(V)$ is equipped with $n^{th}$ products for all integers $n$, and it is immediate that all of these products are $Q$-linear.\ Furthermore any identity satisfied by the products in $V$ carries  over to $Q(V)$.\ Therefore the Jacobi identity is satisfied in $Q(V)$, and of course the vacuum vector is $1$.\ Finally, for similar reasons the creation property holds in $Q(V)$ and therefore $Q(V)$ is a vertex algebra over $Q$.\ Indeed it is straightforward to see that the map $\iota$ of Theorem \ref{p:QMVtensor}
is an isomorphism of vertex algebras over $Q$ where here we regard the commutative algebra $Q$ as a vertex algebra with vertex operators satisfying 
\[Y(f, z)g\df fg\] 
and $Q\otimes V$ is a tensor product of vertex algebras.

Despite the naturalness of this construction it is \emph{not} so useful for us.\ The reason is that it ignores all considerations of the gradations on $Q$, $V$ and $Q(V)$, and these are fundamental aspects of these structures that we want to keep track of.\ Put another way, if we reintroduce the same map with a different name:
 \begin{equation}
    \label{QV=QVQ}
\iota': Q\otimes V \stackrel{\cong}{\longrightarrow} Q(V),\quad\quad f(\tau)\otimes v \mapsto f(\tau)v
\end{equation}
then we here want to regard $Q\otimes V$ as a tensor product of graded vertex algebras and $\iota'$ as a morphism of such objects.\ The difference between $\iota$ and $\iota'$ is that they are morphisms in different categories.

\subsection{\texorpdfstring{$Q^{(1/2)}$}{Q-halved} qua QVOA}

We are going to develop the algebraic structure on $Q(V)\cong Q\otimes V^{(2)}$ that incorporates both the vertex algebra structure and the grading.\ It is convenient to begin with $Q$, or rather $Q^{(1/2)}$ (cf.\ Appendix \ref{SSdouble}).\ Here is the main result of this Subsection.

\begin{thm}\label{thmQVA}
 $Q^{(1/2)}$ is the underlying Fock space for a quasi-vertex operator algebra $$(Q^{(1/2)}, Y, 1, \rho).$$  
 ($Y$ and $\rho$ will be described in the course of the proof.)
\end{thm}
\begin{proof}
Recall that $Q^{(1/2)}$ is nothing but a regrading of $Q$:
$$Q^{(1/2)}= \oplus_{k\geq0} Q^{(1/2)}_k, \quad\quad Q^{(1/2)}_k\df Q_{2k},$$
in particular $Q^{(1/2)}$ is a commutative, associative algebra and it admits a graded \emph{derivation} $\theta \df q\frac{d}{dq}$ which is a raising operator of weight $1$.\ Then there is a  vertex algebra $(Q^{(1/2)}, Y, 1, \theta)$ for which
\begin{equation}    
\label{Ydefn}
Y(f, z)\df e^{z\theta }f=\sum_{n\geq0} \frac{z^n}{n!} \theta^nf \quad\quad (f\in Q^{(1/2)}).
\end{equation}
(Cf.\  Example 1 of Appendix \ref{SScanonical}.)\
Thus $n^{th}$ products in $Q^{(1/2)}$ are such that for all $n\geq0$ we have $f(n)g=0$ and
\[f(-n-1)g \df \frac{1}{n!}(\theta^n f)g.\]
Note that from the last display we see that $f(-2)1= \theta f$, so that $\theta$ is the canonical derivation of the vertex algebra.

\medskip
Next we show that this vertex algebra is $\ZZ$-graded.\ The grading is already on display and we have to show that if $f, g$ have conformal degrees $k$ and $\ell$ respectively then $f(n)g$ has conformal degree $k+\ell-n-1$. This is obvious if $n\geq0$.\ Now take $n\geq0$ and consider
$f(-n-1)g$ as given above.\ Since $\theta$ raises weights by $1$ then
$f(-n-1)g$ has conformal degree $k+\ell+n= k+\ell-(-n-1)-1$ as needed.

\medskip
Now we come to the main point: the Lie algebra representation
$$\rho: \mathfrak{sl}_2\rightarrow \Endo(Q^{(1/2)}).$$

This is well-known in the theory of quasi-modular forms where the action of
$\mathfrak{sl}_2$ on $Q$ is described, for example, in \cite{Zagier123}. In the notation of (loc.\ cit.), $Q=\oplus_k Q_{2k}$ admits
a weight $2$ lowering operator $\delta=\partial/\partial E_2$, a weight $2$ raising operator $D=q\frac{d}{dq}$ and the Euler operator $E$.\ The following relations hold:
$$[D, \delta]=E,\quad [E, D]=2D,\quad  [E, \delta]=-2\delta.$$

If we transfer this action to $Q^{(1/2)}$, replace $D$ with $\theta$, retain the notations $\delta, E$ for the lowering  and Euler operators respectively on $Q^{(1/2)}$ (hopefully causing no confusion) we arrive at the definition of $\rho$, viz.,
\begin{equation}
 \label{rhodef}
  \rho: L(1) \mapsto -\delta,\quad  L(0)\mapsto E,\quad L(-1) \mapsto \theta   
\end{equation}
as prescribed in our definition of quasi-vertex operator algebra.

\medskip
Finally we have to verify the following two axioms for $a\in Q_k^{(1/2)}$:
\begin{enumerate}
    \item[(i)] (Translation covariance).\ $(L(-1)a)(n)= -na(n-1)=[L(-1), a(n)]$,
    \item[(ii)] $[L(1), a(n)] = (L(1)a)(n)+(2k-n-2)a(n+1)$, 
\end{enumerate}
where we are using the notation of  \eqref{rhodef}
We start with (i).\
If $n\geq0$ this follows because $a(n)=0$ whenever $n\geq0$.\ Suppose that $n<0$.\ Then
$$(L(-1)a)(n)= (\theta a)(n)=\tfrac{1}{(-n-1)!} \theta^{-n-1}\theta a=-\tfrac{n}{(-n)!}\theta^na=-na(n-1),$$
which is the first equality in (i).\ As for the second,
$$[L(-1), a(n)]b=\theta(a(n)b)-a(n)\theta b=(\theta a(n))b=(L(-1)a)(n)b,$$
where to get the second equality we used the fact that $\theta$, being the canonical derivation of the vertex algebra $Q^{(1/2)}$,  is a derivation of for all products $a(n)b$.\ This completes the proof of translation covariance.

\medskip
It remains to establish (ii).\ As usual the result is obvious if $n\geq0$, so assume that
$n=-m<0$.\ We have
\begin{eqnarray}\label{calc1}
[L(1), a(n)]b = \tfrac{-1}{(m-1)!}[\delta, \theta^{m-1}a]b=\tfrac{-1}{(m-1)!} \delta(\theta^{m-1}a)b       
\end{eqnarray}
where to get the second equality we use the fact that $\delta$ is a derivation of $Q$.\ Next we need the following result which may be proved by an induction:
\begin{lem} If $a\in Q_{2k}$ then 
$$\delta \theta^ma = -m(2k+m-1) \theta^{m-1}a + \theta^m\delta a. $$
$\hfill\Box$
\end{lem}

With this Lemma we see from \eqref{calc1} that
\begin{align*} 
[L(1), a(n)] &= \tfrac{-1}{(m-1)!}\delta(\theta^{m-1}a)\\
&=\tfrac{-1}{(m-1)!} \left\{ ((-(m- 1)(2k+m-2)\theta^{m-2}a +\theta^{m-1}\delta a\right\}  \\
&=  \tfrac{2k+m-2}{(m-2)!} \theta^{m-2}a - \tfrac{1}{(m-1)!}\theta^{m-1}\delta a \\
&= (2k-n-2)a(n+1)  + (L(1) a)(n).
\end{align*}

This shows that condition (ii) is indeed satisfied and the proof of the Theorem is complete.
\end{proof}

\subsection{Passage to \texorpdfstring{$Q(V)$}{Q(V)}} In this Subsection $Q^{(1/2)}$ denotes the QVOA
described in Theorem \ref{thmQVA}.\ We can form tensor products of QVOAs, cf.\ Appendix \ref{SScat}.\ In particular for a QVOA $V$ we have the QVOA
\begin{align}
\label{QVOATP}
  Q^{(1/2)}\otimes V  &\df (Q^{(1/2)}, Y, 1, \rho)\otimes (V, Y, \mathbf{1}, \rho')\notag\\
  &=(Q^{(1/2)}\otimes V, Y\otimes Y, 1\otimes\mathbf{1}, \rho\otimes\rho').
\end{align}

If we let (hopefully without introducing further confusion) $L(0)$ and $L(\pm 1)$ be the usual Virasoro operators on $V$ then the corresponding Euler, raising and lowering operators on $Q^{(1/2)}\otimes V$ are respectively
\begin{eqnarray}\label{sl2ops}
 && E\otimes \Id + \Id\otimes L(0), \notag \\
 && \theta\otimes \Id + \Id\otimes L(-1),\\
 &&\delta\otimes \Id + \Id\otimes L(1).\notag
\end{eqnarray}

One easily sees that $1\otimes V$ and $Q^{(1/2)}\otimes \mathbf{1}$ are subQVOAs of the tensor product, that the latter is a commutative algebra isomorphic to $Q^{(1/2)}$ with respect to the $-1$ product, and that the latter is a module over the former again with respect to the $-1$ product.\ Thus the perspective of (\ref{QV=QVQ}) is not lost; it can be recovered from the structure of $Q^{(1/2)}\otimes V$ as a QVOA.

\medskip
Finally we simply note that by doubling the grading of $Q^{(1/2)}\otimes V$ we get $Q\otimes V^{(2)}$  which we identify with $Q(V)$ by way of Theorem \ref{p:QMVtensor}.\ Thus we have
\begin{thm}
Let $V$ be a QVOA.\ Then the space
of $V$-valued quasi-modular forms $Q(V)$
has the structure of a doubled quasi-vertex operator algebra.\ $\hfill\Box$
\end{thm}

\section{Some additional operators on \texorpdfstring{$Q(V)$}{Q(V)}}
\label{s:operators}
 The tensor product $Q^{(1/2)}\otimes V$, or $Q\otimes V^{(2)}$, admits more than just the  `diagonal' $\mathfrak{sl}_2$-action used to define the QVOA structure in the previous Section.\ It admits an action of $\mathfrak{sl}_2\oplus\mathfrak{sl}_2$.\ We shall make use of some of these additional operators in the following Sections.\
 They will turn out to be crucial to our cause, which is now to understand the embedding
 $$M(V)\subseteq Q(V).$$
We have already seen that $Q(V)$ is closed under the following derivative operators:
\begin{align*}
    \theta &= qd/dq,\\
    D_k &= \theta - \tfrac{k}{12}E_2,\\
    \nabla_k &= 2\pi i D_k +L(-1),
\end{align*}
though only the last of these leaves $M(V)$ invariant, cf. Subsection \ref{ss:nablaproof}. All of these operators raise conformal degrees by $2$.

\subsection{The lowering operator and the main Theorem}
 \label{ss:lowering}
 We define a lowering operator $\Lambda$ on $Q\otimes V^{(2)}$ as follows:
$$\Lambda = \frac{12}{2\pi i}\delta\otimes \Id+\Id\otimes L(1).$$

\medskip
In this Section we will mainly be viewing $Q(V)$ as functions on the upper half-plane. Then $\Lambda$ can be expressed more simply as:
\begin{equation}
    \label{eq:loweringoperator}
    \Lambda = \frac{12}{2\pi i} \delta + L(1) = \frac{12}{2\pi i }\frac{\partial}{\partial E_2} +L(1).
\end{equation}

Among our main results are the next Theorem and its Corollary, which provides a very useful characterization of $M(V)$ as the kernel of $\Lambda$.
\begin{thm}\label{thmXLambda}
Suppose that $f(\tau)\in Q_k(V)$.\ Then
\begin{eqnarray}\label{Lambdaact}
f||_k\gamma(\tau) = \exp(X\Lambda)f(\tau).
\end{eqnarray}
\end{thm}

Before proving this theorem, let us first record some important corollaries:
\begin{cor}\label{cokerL} We have
$$M(V)=\ker\Lambda.$$
\end{cor}
\begin{proof}
By definition, $f\in M(V)$ if, and only if, $f||_k\gamma(\tau)=f$.\ So by \eqref{Lambdaact} $f\in M(V)$ if, and only if, $\exp(X\Lambda)f(\tau)=f(\tau)$.\ Since $\Lambda$ lowers weights by $2$ this last equality holds if, and only if, $\Lambda f=0$ and the Corollary is proved.
\end{proof}
\begin{cor}
    Suppose that
    $$f||_k\gamma = \sum_{n\geq 0} X^nQ_n(f)$$
    for $V$-valued holomorphic functions $Q_n(f)$.\ Then
    \begin{eqnarray}\label{Qform}
     Q_n(f)= \frac{1}{n!}\Lambda^n(f).   
    \end{eqnarray}
\end{cor}

\begin{rmk}
 Equation \eqref{Qform} is the VOA analog of the well-known result for scalar-valued quasi-modular forms that if $g\in Q_k$ satisfies
    $g|_k\gamma = \sum_{n\geq0} X^nQ_n(f)$ then we have
    $$Q_n(f)=\frac{1}{n!}\partial^n f/\partial E_2^n.$$
\end{rmk}

\begin{proof}[Proof of Theorem \ref{thmXLambda}:]
Since $f\in Q_k(V)$ with $k=2n$ then there is an equality of the form
\begin{eqnarray}\label{fdef}
f=\sum_{m=0}^n E_2^m\sum_{\ell=0}^{n-m} f_{m,\ell}v_{m, \ell},
\end{eqnarray}
where $f_{m, \ell}\in M_{2\ell}$ and $v_{m, \ell}\in V_{n-m-\ell}$.\ Let us write $j$ in place of $j(\gamma, \tau)$.\ Then
\begin{align*}
  f||_k\gamma(\tau) &=j^{-k}K(\gamma, \tau)^{-1}f(\gamma\tau)\\
  &=j^{-k+2L(0)}e^{cjL(1)}f(\gamma\tau)\\
  &=j^{-k+2L(0)}e^{cjL(1)}\sum_{m, \ell}E_2^m(\gamma\tau)f_{m, \ell}(\gamma\tau)v_{m, \ell}\\
  &=j^{-k+2L(0)}\sum_{r\geq 0} \frac{c^rj^r}{r!}L(1)^r\sum_{m, \ell}E_2^m(\gamma\tau)f_{m, \ell}(\gamma\tau)v_{m, \ell}\\
  &=j^{-k+2L(0)}\sum_{r\geq 0} \frac{c^rj^r}{r!}\sum_{m, \ell}E_2^m(\gamma\tau)f_{m, \ell}(\gamma\tau)L(1)^rv_{m, \ell}\\
  &=\sum_{r\geq0}\sum_{m, \ell}j^{-k+2(n-m-\ell-r)+r}\frac{c^r}{r!}E_2^m(\gamma\tau)f_{m, \ell}(\gamma\tau)L(1)^rv_{m, \ell}\\
  &=\sum_{m, \ell}j^{-2m-2\ell-r}\sum_{r\geq 0} \frac{c^r}{r!}E_2^m(\gamma\tau)f_{m, \ell}(\gamma\tau)L(1)^rv_{m, \ell}\\
  &=\sum_{r\geq0}j^{-r} \frac{c^r}{r!}\sum_{m, \ell}(j^{-2}E_2(\gamma\tau))^m(j^{-2\ell}f_{m, \ell}(\gamma\tau))L(1)^rv_{m, \ell}\\
  &=\sum_{r\geq0} \frac{(XL(1))^r}{r!}\sum_m \left(E_2(\tau)+\frac{12X}{2\pi i}\right)^m \sum_{\ell}f_{m, \ell}(\tau)v_{m, \ell}.
\end{align*}

On the other hand we have
\begin{align*}
\exp(X\Lambda)f(\tau)&=\sum_{s\geq0} \frac{X^s\Lambda^s}{s!}
\sum_{m, \ell} E_2^mf_{m, \ell} v_{m, \ell}\\
&=\sum_{s\geq0} \frac{1}{s!}\left(\frac{12X}{2\pi i}\partial/\partial E_2 +XL(1)\right)^s\sum_{m, \ell} E_2^mf_{m, \ell} v_{m, \ell}\\
&=\sum_{s\geq0} \frac{1}{s!}\sum_{t=0}^s \binom{s}{t}\left(\frac{12X}{2\pi i}\partial/\partial E_2 \right)^{s-t}
(XL(1))^{t}\sum_{m, \ell} E_2^mf_{m, \ell} v_{m, \ell}\\
&=\sum_{s\geq0} \frac{1}{s!}\sum_{t=0}^s \binom{s}{t} \sum_{m, \ell}\frac{m!}{(m-s+t)!}\left(\frac{12X}{2\pi i}\right)^{s-t}
(XL(1))^{t}E_2^{m-s+t}f_{m, \ell} v_{m, \ell}\\
&=\sum_{s\geq0}\sum_{m, \ell, t}\frac{1}{t!}\binom{m}{s-t}
\left(\frac{12X}{2\pi i}\right)^{s-t}(XL(1))^{t}E_2^{m-s+t}f_{m, \ell} v_{m, \ell}\\
&=\sum_{t\geq0}\sum_{m, \ell}\frac{(XL(1))^t}{t!}
\left(E_2 +\frac{12X}{2\pi i}    \right)^m f_{m, \ell}v_{m, \ell}.
\end{align*}

A comparison of these expressions completes the  proof of the
Theorem.
\end{proof}

\subsection{Surjectivity of the lowering operator}
We address the question of whether $\Lambda:Q(V)\rightarrow Q(V)$ is a surjective map. This is the analog of asking whether $L(1)$ is a surjective map on $V$. The answer for $V$, described in \cite{DLMsl2} and recalled below in Lemma \ref{lemL1surj}, depends on technical properties of $V$.

\medskip
We start with a simple observation.
\begin{lem}\label{lemineq} For any QVOA $V$ we  have
$$\dim Q_{2k+2}(V)-\dim Q_{2k}(V) \leq \dim M_{2k+2}(V).$$
Equality holds if, and only if,
$$\Lambda:Q_{2k+2}(V)\rightarrow Q_{2k}(V)$$
is a surjection.
\end{lem}
\begin{proof}
  Since $\Lambda$ lowers weights by $2$ then there is a linear isomorphism
  $$\Lambda(Q_{2k+2}(V))\cong Q_{2k+2}(V)/M_{2k+2}(V).$$
  This holds because, by Corollary \ref{cokerL}, $M_{2k+2}(V)=Q_{2k+2}(V)\cap \ker\Lambda.$\ Thus
  \begin{eqnarray*}
   \dim Q_{2k}(V)\geq \dim \Lambda(Q_{2k+2}(V))=\dim Q_{2k+2}(V)-\dim M_{2k+2}(V),   
  \end{eqnarray*}
  and both parts of the Lemma now follow.
\end{proof}

As preparation for the next result we state the Proposition from Section 2.4 of \cite{DLMsl2} mentioned above.\ Although the result likely holds for any QVOA, we state it only when $V$ is a VOA because that is the assumption of  \textit{loc.\ cit.}
\begin{lem}\label{lemL1surj} For any VOA $V$, the map
$$L(1): V_{k+1}\rightarrow V_k$$
is a \textit{surjection} unless perhaps $k=0$.
$\hfill \Box$
\end{lem}

The first main result of this Subsection is the $Q(V)$-analog of this Lemma.
\begin{thm}\label{thmalmostonto} For any VOA $V$, the map
$$\Lambda: Q_{2k+2}(V)\rightarrow Q_{2k}(V)$$
is a surjection unless perhaps $k=0$.
\end{thm}
\begin{proof} We have
$$Q_{2k}(V)=\bigoplus_{r=0}^k Q_{2r}\otimes V_{k-r}.$$
Fix any $k\neq 0$.\ We show by induction on $r$ that $\im\Lambda\supseteq Q_{2r}\otimes V_{k-r}$.\ The first case is when $r=0$, when we have to show that $V_k\subseteq\im\Lambda$.\ We have
\begin{align*}
\Lambda(V_{k+1})&=\left(\frac{12}{2\pi i}\partial/\partial E_2\otimes L(1)\right)(V_{k+1})\\
&=L(1)V_{k+1}\\
&=V_k
    \end{align*}
    where the last equality follows from Lemma \ref{lemL1surj}.\ Thus indeed
    $V_k\subseteq \im\Lambda$.

\medskip
    Proceeding by induction, choose $f\in Q_{2r}$ and $v\in V_{k-r}$.\ Let $w\in V_{k-r+1}$ satisfy $L(1)w=v$.\ Such a $w$ exists by Lemma \ref{lemL1surj} as long as  $k\neq r$.\ Let us assume this for now.\ Then
    \begin{eqnarray*}
     &&\Lambda(f\otimes w)=\frac{12}{2\pi i}\partial f/\partial E_2\otimes w+ f\otimes v 
    \end{eqnarray*}
    By induction the first term on the right side lies in $\im\Lambda$.\ Therefore $f\otimes v\in\im\Lambda$ whence
    $Q_{2r}\otimes V_{k-r}\subseteq\im\Lambda$.

    \medskip
    It remains to consider the case $k=r$, when we have to show that $Q_{2k}\subseteq\im\Lambda$.\ But it is clear that $\Lambda(Q_{2k+2})=Q_{2k}$ and this completes the induction, and with it the proof of the Theorem.
\end{proof}

\begin{cor} 
Suppose that $V$ is of CFT-type (i.e., $V_0=\CC\mathbf{1}$).\ Then $\Lambda: Q(V)\rightarrow Q(V)$ is a surjection.
\end{cor}
\begin{proof}Consider the map $\Lambda:Q_2(V)\rightarrow Q_0(V)$.\ We have
\begin{eqnarray*}
&&\Lambda(Q_2(V))\supseteq \left(\frac{12}{2\pi i}\partial/\partial E_2\otimes 1+1\otimes L(1) \right) (\langle E_2\otimes \mathbf{1}\rangle)   =\CC\mathbf{1}.
\end{eqnarray*}
Because $V_0=\CC\mathbf{1}$ by assumption it follows that
$Q_0(V)=\CC\mathbf{1}$ and so $\Lambda$ surjects $Q_2(V)$ onto $Q_0(V)$.\ Therefore the statement of the Corollary follows by Theorem \ref{thmalmostonto}.
\end{proof}

\begin{cor}\label{cor710}
For any VOA $V$ we have
$$\dim Q_{2k+2}(V)-\dim Q_{2k}(V) = \dim M_{2k+2}$$
unless perhaps $k=0$.\ If $V$ is of CFT-type then the equality holds for all $k$.
$\hfill\Box$
\end{cor}

\begin{rmk}
It is well-known, cf. Lemma 5.2 of \cite{DMshifted}, that if $V$ is of CFT-type i.e., $V_0=\CC\bone$, then $V_k=0$ for $k<0$.
\end{rmk}
\subsection{The weight-preserving operator \texorpdfstring{$P$}{P}}\label{SSP}

It is defined as follows: as an endomorphism of $Q\otimes V^{(2)}$, it is described as
\begin{align}
\label{Pdef}
    P\df \exp\left(-\frac{2\pi i}{12} E_2 \otimes L(1)\right).
\end{align}
That is, on a pure tensor $f\otimes v$ we have
\[
P(f\otimes v) = \sum_{n\geq 0}\frac{1}{n!}\left(\left(-\frac{2\pi i}{12} E_2\right)^n f\right) \otimes \left(L(1)^nv\right).
\]

Similarly to $\Lambda$, $P$ belongs to the natural $\mathfrak{sl}_2\oplus\mathfrak{sl}_2$ acting on $Q(V)$ but \emph{not} the diagonal $\mathfrak{sl}_2$.\ We shall also use $P$ to denote the composition of \eqref{Pdef} with $\iota$.\ This should cause no confusion.
It is clear that $P$ is an invertible weight-preserving map 
\[
P\colon Q \otimes V^{(2)} \stackrel{\cong}{\longrightarrow} Q(V).
\]

Now we come to the main result of this Subsection.\ We state it only for QVOAs of CFT-type, however the interested reader may readily formulate the general case.

\begin{thm}\label{thmPiso}Suppose that $V$ is QVOA of CFT-type.\ Then the following diagram commutes
\[
    \begin{xymatrix}{
    Q\otimes V^{(2)}\ar[rr]^{P}_{\cong}&& Q(V)\\
    M\otimes V^{(2)}\ar[rr]^{P|_{M\otimes V^{(2)}}}_{\cong}\ar@{^{(}->}[u]&& M(V)\ar@{^{(}->}[u]\\
    }
\end{xymatrix}
\]
where the vertical maps are the natural containments.
\end{thm}
\begin{proof} 
The main point of the Theorem is that restricting $P$ to $M\otimes V^{(2)}$  induces a surjection onto $M(V)$.\ Since $P$ preserves weights it suffices to show that $P|_{M\otimes V^{(2)}}$ induces an isomorphism at weight $k=2n$.\ 

Let $f\in Q_k(V)$ be as in \eqref{fdef} and assume that in fact $f\in M_k(V)$.\ Then $\Lambda f=0$ by Corollary \ref{cokerL}.\ Writing this out explicitly shows that
\begin{align*}
0 &= \left( \frac{12}{2\pi i}\partial/\partial E_2\otimes 1+1\otimes L(1)  \right).\sum_{m=0}^nE_2^m\sum_{\ell}^{n-m} f_{m, \ell}v_{m, \ell}\\
&=\frac{12}{2\pi i}\sum_{m=0}^nmE_2^{m-1}\sum_{\ell}^{n-m} f_{m, \ell}v_{m, \ell} + \sum_{m=0}^n E_2^m\sum_{\ell}^{n-m}f_{m, \ell}L(1)v_{m, \ell}.
\end{align*}
    Now fix $m\geq 1$ and identify coefficients of $E_2^{m-1}$ to see that
    $$\sum_{\ell}^{n-m}f_{m, \ell}v_{m. \ell}=-\frac{2\pi i}{12m}L(1)\sum_{\ell}^{n-m+1}f_{m-1, \ell}v_{m-1, \ell}.$$
    It follows that for all $m\geq 0$ we have
    $$\sum_{\ell}^{n-m}f_{m, \ell}v_{m, \ell}=\frac{(-2\pi i/12 L(1))^m}{m!}\sum_{\ell}^{n} f_{0, \ell}v_{0, \ell},$$
    and hence that
    $$f=\sum_{m=0}^n \frac{(-2\pi i/12 E_2 L(1))^m}{m!}\sum_{\ell}^{n} f_{0, \ell}v_{0, \ell}.$$

    Now $v_{0, \ell}$ has conformal weight $n-\ell$ in $V$.\ Therefore
    $L(1)^{n+1}\in V_{-\ell-1}$ and this vanishes because $\ell\geq0$
    and $V$ has no nonzero negative weight states because it is of CFT-type.\ Therefore the last displayed equation may be rewritten in the form
    $$ f= P\sum_{\ell=0}^n f_{0, \ell}v_{0, \ell}.$$

This shows that  
$P(M\otimes V^{(2)})\supseteq M(V)$.\
On the other hand, if we start with a state $\sum_{\ell=0}^n f_{0, \ell} v_{0, \ell}\in M(V)$ and run the argument above backwards we see that
$P\left(\sum_{\ell=0}^n f_{0, \ell} v_{0, \ell}\right)\in \ker\Lambda$.\ 
Therefore $P|_{M\otimes V^{(2)}}$ is a surjection onto $M(V)$, and the Theorem is proved.
\end{proof}

This allows us to compute the dimension of $M_k(V)$ giving a more explicit computation compared to the previous result.\ Again we assume for simplicity that $V$ is a VOA.\  We have via an isomorphism given by the above exponential that

\begin{cor}
If $V$ is a VOA of CFT-type then 
\begin{align*}
\dim M_k(V)=\sum_{i=0}^{k/2}\dim M_{2i}\dim V_{\frac{k}{2}-i}\end{align*}
\end{cor}

\begin{ex}
    Consider the case when $V=S$ is the rank-one Heisenberg VOA generated with canonical weight $1$ state $h$ (cf.\ 
    Appendix \ref{AppS}).\  A basis for $(M\otimes V^{(2)})_4$
is 
$$E_4\otimes \mathbf{1},\quad 1\otimes h(-1)^2\mathbf{1},\quad 1\otimes h(-2)\mathbf{1}.$$
The respective $P$-image of these three states are 
$$E_4\mathbf{1},\quad h(-1)^2\mathbf{1},\quad h(-2)\mathbf{1}-\frac{2\pi i}{6}E_2h,$$
which is indeed a basis of $M_4(S)$ as we see from Table \ref{tab:bases} below.
\end{ex} 

\subsection{Some graded dimensions}
Theorem \ref{thmPiso} makes it easy to read-off the graded dimensions (Hilber-Poincare series) of spaces such as $Q(V)$ and $M(V)$.\ For a QVOA $V=\oplus_n V_n$ we write
$$\dim_q V\df \sum_n \dim V_n q^n.$$

\begin{lem}\label{lemQM}
Suppose that $V$ is a QVOA of CFT-type.\ Then
\begin{enumerate}
    \item[(a)] \[\sum_{k\geq0} \dim Q_{2k}(V)q^k=\frac{\dim_qV}{(1-q)(1-q^2)(1-q^3)}\]
    \item[(b)] \[\sum_{k\geq0} \dim M_{2k}(V)q^k=\frac{\dim_qV}{(1-q^2)(1-q^3)}.\]
\end{enumerate}
\end{lem}
\begin{proof}
Recall that the Hilbert-Poincare series for $Q$ and $M$ are (with the standard grading)
$$\frac{1}{(1-q^2)(1-q^4)(1-q^6)},\quad \frac{1}{(1-q^4)(1-q^6)}$$
respectively.\ Since there is a graded isomorphism $Q(V)= Q\otimes V^{(2)}$ then part (a) follows.\ Part (b) is similar, except we need Theorem \ref{thmPiso} to identify $M(V)$ as a $\ZZ$-graded linear space.
\end{proof}
 
\begin{cor}
   The sequence $\{\dim M_{2k}{V}\}$ is the first difference sequence of $\{\dim Q_{2k}(V)\}$. 
\end{cor}
\begin{proof}
This is because from Lemma \ref{lemQM} we have
$$ (1-q)\sum_{k\geq0} \dim Q_{2k}(V)q^k= \sum_{k\geq0} \dim M_{2k}(V)q^k.$$
\end{proof}

Finally we again specialize to the case that $V=S$ is the rank-one Heisenberg VOA.\ We present a strange and unexpected formula for the first difference sequence of
$\{\dim M_{2k}(S)\}$. This result, first observed computationally, proved useful for discovering some of the structural results on $Q(V)$ and $M(V)$ established above. We have not found an explicit combinatorial bijection that explains this result.
\begin{lem} We have
$$\dim M_{2k}(S)-\dim M_{2k-2}(S)= \mbox{$\#$ doubletons in all partitions $\lambda \vdash k+2$}.$$
\end{lem}
\begin{proof}
   We have
   $$\dim_q S=\sum_{n\geq0} p(n)q^n = \varphi(q):= \prod_{n\geq 1} (1-q^n)^{-1}.$$
   Therefore by Lemma \ref{lemQM}(b) we see that the first difference sequence of $\{\dim M_{2k}(S)\}$ is equal to
   $$ \frac{1}{(1+q)(1-q^3)\varphi(q)}.$$

   Now the Lemma follows from the generating function for the number of doubletons in all partitions of $n$, which is
   known (OEIS A116646) to be
   $$ \frac{q^2}{(1+q)(1-q^3)\varphi(q)}.$$
\end{proof}

\begin{ex}
We explain how the constructions above appear in the basic example of the rank-one Heisenberg algebra $S = \CC[h(-1),h(-2),\ldots]$. From the definitions above, we have defined 
\[Q(S)=\CC[E_2,E_4,E_6,h(-1),h(-2),...]\] 
and with a grading where $E_2$ has weight $k=2$, $E_4$ has weight $k=4$, $E_6$ has weight $k=6$, and $h(-n)$ has weight $k=2n$. We let $Q_k(S)$ denote the $k$th graded piece. Because of the presence of quasimodular forms, we can view elements of $Q(S)$ as functions on $\uhp$.
\begin{thm}
        The dimension $\dim Q_k(S)$ counts the number of ways of writing $k$ as a sum of even integers, but where the parts 2, 4 and 6 come in two colors. Equivalently, for $k=2n$, $\dim Q_k(S)$ is thus the number of colored partitions of $n$ for which the parts $1, 2, 3$ have two colors.\ In particular, this is sequence $A000098$ in the OEIS.
        \begin{proof}
            This follows immediately from the grading of $Q(S)$ by considering the basis of monomials. The monomial $E_2^aE_4^bE_6^c (\prod_{n\geq 1} h(-n)^{d_n})$, where $d_n = 0$ for almost all $n$, corresponds to the partition $k= 2a+4b+6c+\sum_{n\geq 1} 2nd_n$.
        \end{proof}
    \end{thm}

Recall that above we explained how the $S$-valued modular forms $M_k(S)\subseteq Q_k(S)$ are the kernel of the lowering operator \eqref{eq:loweringoperator}. 
We list the dimension counts $\dim Q_k(S)$ and $\dim M_k(S)$ for small values of $k$ in Table \ref{tab:dimensions}, as well as corresponding bases for $M_k(S)$ for small values of $k$ in Table \ref{tab:bases}.
    \begin{table}
    \centering
        \begin{tabular}{|c|c|c|}
             \hline
             $k$ & $\dim Q_k(S)$ & $\dim M_k(S)$ \\
             \hline
             $0$ & $1$ & $1$  \\
             \hline
             $2$ & $2$ & $1$  \\
             \hline
             $4$ & $5$ & $3$ \\
             \hline
             $6$ & $10$ & $5$ \\
            \hline
             $8$ & $19$ & $9$ \\
             \hline
             $10$ & $33$ & $14$ \\
             \hline
             $12$ & $57$ & $24$ \\
             \hline
             $14$ & $92$ & $35$ \\
             \hline
             $16$ & $147$ & $55$ \\
             \hline
             $18$ & $227$ & $80$ \\
             \hline
             $20$ & $345$ & $118$ \\
             \hline
             $22$ & $512$ & $167$ \\
             \hline
             $24$ & $752$ & $240$ \\
             \hline
             $26$ & $1083$ & $331$ \\
             \hline
             $28$ & $1545$ & $462$ \\
             \hline
        \end{tabular}
        \caption{Dimensions for Heisenberg-valued quasi-modular and modular forms}
        \label{tab:dimensions}
    \end{table}
    \SetTblrInner{rowsep=7pt}
    \begin{table}    
    \centering
        \begin{tblr}{|c||l|}
            \hline
             $k$ & Basis Vectors for $M_{k}(S)$ \\
             \hline \hline
              $0$ & $1$ \\
             \hline
             $2$ & $h(-1)$ \\
             \hline
             $4$ & $E_4, \hspace{6pt} h(-1)^2,\hspace{6pt} E_2h(-1)-\dfrac{3}{\pi i}h(-2)$ \\
             \hline
             $6$ & \makecell[l]{$E_6,\hspace{6pt} E_4h(-1),\hspace{6pt} h(-1)^3$, \\
            $E_2^2h(-1)-\dfrac{6}{\pi i}E_2h(-2)-\dfrac{12}{\pi^2}h(-3)$, \\ $E_2h(-1)^2-\dfrac{3}{\pi i}h(-2)h(-1)$} \\
            \hline
            $8$ & \makecell[l]{$E_4^2,\hspace{6pt} E_6h(-1),\hspace{6pt} E_4h(-1)^2,\hspace{6pt} h(-1)^4,E_4h(-2)-\dfrac{\pi i}{3}E_4E_2h(-1),$\\$h(-4)+\dfrac{\pi^3 i}{54}E_2^3h(-1)-\dfrac{\pi}{6}E_2^2h(-2)-\dfrac{2\pi i}{3}E_2h(-3)$,\\$h(-3)h(-1)-\dfrac{\pi^2}{12}E_2^2h(-1)^2-\dfrac{\pi i}{2}E_2h(-2)h(-1)$, \\$h(-2)^2-\dfrac{\pi^2}{9}E_2^2h(-1)^2-\dfrac{2\pi i}{3}E_2h(-2)h(-1)$, \\ $h(-2)h(-1)^2-\dfrac{\pi i}{3}E_2h(-1)^3$} \\
            \hline
        \end{tblr}
        \caption{Bases for spaces of Heisenberg-valued modular forms $M_{k}(S)$ for $k \in \lbrace 0,2,4,6,8 \rbrace$}
            \label{tab:bases}
    \end{table}
\end{ex}

\subsection{The connection revisited}

Recall that $\nabla_{k}$ is a graded operator given by the formula
\begin{align*}
    \nabla_{k} = 2\pi i D_{k} + L(-1) = 2\pi i \theta - \tfrac{2\pi i k}{12}E_{2} + L(-1).
\end{align*}
It was also shown in section \ref{s:connection} that $\nabla_{k}$ maps modular forms to modular forms, increasing weights by $2$.\ The goal of this subsection is to show that the operators $\nabla_k$ can be used to describe a decomposition of the space $M_{k}(V)$. We require the following result:
\begin{thm}
\label{t:quasiE2}
Let $V$ be a QVOA of CFT-type. Fix a weight $k \geq 0$ and write $f \in Q_{k}(V)$ in the form
\begin{align*}
            f = \sum_{i+j=k} \alpha_{i} \beta_{j}
        \end{align*}
        where $\alpha_{i}$ has modular weight $i$ and $\beta_{j} \in V$ has $L(0)$-weight $(k-i)/2$ in $V$.
    Suppose that $\alpha_{i} \in M_{i}$ for all $0 \leq i \leq k$ even. Then $f \in M_{k}(V)$ if and only if $f$ is quasi-primary.
    \end{thm}
\begin{proof}
        A simple computation first shows that
        \begin{align*}
            \tau^{k-2L(0)} f = \sum_{i+j=k} \tau^{i} \alpha_{i} \beta_{j} = f(-1/ \tau).
        \end{align*}
        Suppose first that $f$ is quasi-primary, that is, $L(1)f = 0$. Using this, we get
        \begin{align*}
            e^{-\tau L(1)} \tau^{k-2L(0)} f( \tau ) = e^{- \tau L(1)} \left( \sum_{i+j=k} \tau^{i} \alpha_{i} \beta_{j} \right) = \sum_{i+j=k} \tau^{i} \alpha_{i} \beta_{j}
        \end{align*}
        which is indeed equal to $f( -1/\tau )$ from the previous display, and hence $f \in M_{k}(V)$. Now if $f \in M_{k}(V)$, then
        \begin{align*}
            \sum_{i+j=k} \tau^{i} \alpha_{i} \beta_{i} = \left( \sum_{\ell \geq 0} \frac{(-1)^{\ell}\tau^{\ell}}{\ell!} L(1)^{\ell} \right) \cdot \sum_{i+k=k} \tau^{i} \alpha_{i} \beta_{j}.
        \end{align*}
        Expanding this expression in $\ell$, we see that we must have
        \begin{align*}
            L(1) \sum_{i+j=k} \tau^{i} \alpha_{i} \beta_{j} = \left( \sum_{\ell \geq 2} \frac{(-1)^{\ell}\tau^{\ell}}{\ell!} L(1)^{\ell} \right) \cdot \sum_{i+j=k} \tau^{i} \alpha_{i} \beta_{j}
        \end{align*}
        Notice that since applying $L(1)$ to an element of $Q_{k}(V)$ never introduces factors of $\tau$, the above expression holds only when $L(1)f = 0$, which establishes the theorem.
\end{proof}

In what follows, we denote by $M_{k}(QP)$ the set of elements in $M_{k}(V)$ which lie in the kernel of $L(1)$, and by $QP_{n}(V)$ the set of elements in $V$ which lie in the kernel of $L(1)$.  We also realize $M_k$ as a subspace of $M_k(V)$ by multiplying with the vacuum.

\begin{thm}
 Suppose $V$ is a VOA of CFT-type. Then,
\begin{align*}
 M_k(V)/M_k(QP)=\nabla_{k-2}(M_{k-2}(V)/M_{k-2})\oplus R_k
 \end{align*}
 therefore we obtain
\begin{align*}
M_k(V)=M_k(QP)\oplus \nabla_{k-2}(M_{k-2}(V) / 
M_{k-2})\oplus R_k
\end{align*}
Moreover $\dim R_k=\dim M_{k-2}(\dim V_1-\dim \ker L(1)\vert_{V_1})$. In particular if $V$ has a nondegenerate invariant form then $R_k=0$. 
\end{thm}
\begin{proof}
Since $\nabla_k = 2\pi i D_k +L(-1)$ we have that $\nabla_k(M_k)\subseteq M_{k+2}(QP)$ for any $k$. Therefore, $\nabla_k\colon M_{k}(V)/M_k\rightarrow M_{k+2}(V)/M_{k+2}(QP)$ is well defined. An easy computation shows that
\begin{align*}
    \ker \nabla_{k} = \begin{cases}
        0 & \text{if $k \neq 12n$ for $n \geq 1$} \\ \CC[ \Delta^{n} ] & \text{if $k = 12n$ for $n \geq 1$}
    \end{cases}
\end{align*}
where we recall that $\Delta$ here denotes Ramanujan's delta function. This follows from the fact that the Serre derivative annihilates the space $\CC [ \Delta^{n} ]$ for all $n \geq 1$, hence $\ker \nabla_{k} \subseteq M_{k}$. Now for any $f \in M_{k}(V)$, the image $\nabla_{k}(f)$ is in the kernel of $L(1)$ if and only if $f \in M_{k}$. This is because if $f \in Q_{k}(V)$ has some non-trivial VOA component (i.e. if $L(0)f \neq 0$), the connection introduces a factor of $E_{2}$, and by Theorem \ref{t:quasiE2} such elements cannot be quasi-primary. As a mapping of vector spaces, we see that $\nabla_{k} : (M_{k}(V) / M_{k}) \to (M_{k+2}(V)/M_k(QP))$ is injective. We deduce that
\begin{align}
    M_k(QP)\oplus \nabla_{k-2}(M_{k-2}(V) / M_{k-2}) \label{connectiondecompeq}
\end{align}
gives a subspace of $M_k(V)$. Let us now compute its dimension piece by piece. Firstly we obtain
\begin{align*}
    \dim M_k(QP)&=\sum_{i=0}^{\frac{k}{2}}\dim M_{2i}\dim QP(V)_{\frac{k}{2}-i}
\end{align*}
In the sum above, notice that when $i=\frac{k}{2}$ we have $\dim QP(V)_{\frac{k}{2}-i}=1$ by assumption. Moreover if $i<\frac{k}{2}-1$ then by \cite[Proposition 3.4 (iii)]{DLMsl2} we have
\begin{align*}
    \dim QP(V)_{\frac{k}{2}-i}=\dim V_{\frac{k}{2}-i}-\dim V_{\frac{k}{2}-i-1}
\end{align*}
Hence
\begin{align*}
    \dim M_k(QP)=\dim M_k&+\dim M_{k-2} \cdot \dim QP(V)_1 \\&+\sum_{i=0}^{\frac{k}{2}-2}\dim M_{2i}(\dim V_{\frac{k}{2}-i}-\dim V_{\frac{k}{2}-i-1})
\end{align*}
We also have
\begin{align*}
    \dim \nabla_{k-2}(M_{k-2}(V) / M_{k-2})&=\dim M_{k-2}(V)-\dim M_{k-2} \\
    &= \sum_{i=0}^{\frac{k}{2}-1}\dim M_{2i}\dim V_{\frac{k}{2}-i-1}-\dim M_{k-2}
\end{align*}
and so the dimension of the space (\ref{connectiondecompeq}) is given by
\begin{align*}
    \dim M_k+\dim M_{k-2} \cdot \dim QP(V)_1&+\sum_{i=0}^{\frac{k}{2}-2}\dim M_{2i}(\dim V_{\frac{k}{2}-i}-\dim V_{\frac{k}{2}-i-1}) \\
&+\sum_{i=0}^{\frac{k}{2}-1}\dim M_{2i}\dim V_{\frac{k}{2}-i-1}-\dim M_{k-2}
\end{align*}
which re-arranges to the formula
\begin{align*}
\dim M_k+\dim M_{k-2}\dim QP(V)_1+\sum_{i=0}^{\frac{k}{2}-2}\dim M_{2i}\dim V_{\frac{k}{2}-i}
\end{align*}
As $\sum_{i=0}^{\frac{k}{2}}\dim M_{2i}\dim V_{\frac{k}{2}-i}=\dim M_k(V)$ we see that the difference is precisely
\begin{align*}
    \dim M_{k-2}(\dim V_1-\dim \ker L(1)\vert_{V_1})
\end{align*}
So if $V_1$ is generated by quasi-primary elements, the image of $\nabla$ and $M_k(QP)$ decompose $M_k(V)$.
\end{proof}
For example, if $V$ is CFT-type and QP generated in the sense of Li (see \cite{LiInvariant}) then \cite[page 10]{DLMsl2} gives that
\[ M_k(V)/M_k(QP)=\nabla_{k-2}(M_{k-2}(V)/M_{k-2}).\]

\section{Hecke operators}
\label{s:hecke}
We introduce Hecke operators for $M(V)$ in analogy with the classical Hecke operators. We explain that the systems of Hecke eigenvalues that arise are precisely the same systems as for classical (quasi)-modular forms of level one. The only thing that depends on the QVOA $V$ are the multiplicities of these eigensystems, which are given by the graded dimensions of $V$.

\subsection{Classical Hecke operators}
First we review the classical Hecke operators $T_m$.\
The standard action of $T_m$ on $f\in M_{2k}$ is as follows:
$$T_mf\df m^{k-1}\sum_{[\alpha]} f|_{2k}\alpha.$$
Here, $[\alpha]$ ranges over the orbits in $\Gamma \setminus{D_m}$ where $D_m$ is the set of $2\times 2$ integral matrices of determinant $m$.  We may without loss choose the representatives to be the set of matrices 
$$\twomat ab0d,\quad\quad ad=m,\quad 0\leq b < d.$$
Then we compute
\[
(T_mf)(\tau)
 =m^{2k-1}\sum_d d^{-2k}f(\alpha\tau).
\]

The induced action of $T_m$ on $q$-expansions is well-known.\ If
$f=\sum_n b_nq^n$ is a modular form of weight $2k$ then
\begin{equation}
    \label{defTmextend}
  (T_mf)(q)=\sum_n\sum_{a\mid(m, n)}
a^{2k-1}b_{mn/a^2} q^n.
\end{equation} 
Following \cite{movasati}, we extend the meaning of $T_m$ using equation \eqref{defTmextend} so that it is an operator on the $q$-expansions of  quasi-modular forms of weight $2k$, not just modular forms.

\begin{rmk}\label{rmkTm}
Equation \eqref{defTmextend} applies in the case of a constant $q$-expansion of weight $0$, when it says that
$$T_m(1) =m^{-1}\sigma_1(m).$$
\end{rmk}
\subsection{The operators \texorpdfstring{$T'_m$}{T'm}}
Now we turn to the action of the $m^{th}$ Hecke operator acting on $M(V)$.\ We denote this by $T_m'$ to distinguish it from $T_m$.\ If $f\in M_{2k}(V)$ then we define
\[
T'_m f(\tau)\df m^{k-1}\sum_{[\alpha]} f||_{2k}\alpha(\tau).
\]
Because the double slash operator $f||_{2k}\alpha$ defines a right group action of $\GL_2(\RR)^+$, this is well-defined and defines an operator on $M_{2k}(V)$.

\medskip
On the other hand we may extend $T_m$ to $Q\otimes V^{(2)}=Q(V)$ componentwise, i.e.,
\begin{equation}\label{Tmextend}
T_{m}\left(\sum g_{2r}\otimes v_{\ell}   \right) \df \sum T_{m}g_{2r}\otimes v_{\ell}\quad\quad (g_{2r}\in Q_{2r}, v_{\ell}\in V_{\ell}),
\end{equation}
where, as indicated, we continue to use the symbol $T_m$ and $T_mg_{2r}$ is the action (\ref{defTmextend}) on $q$-expansions.

\begin{thm}\label{thmHecke} 
Both of the operators $T_m$ and $T'_m$ preserve the subspace $M(V)\subseteq Q(V)$, and we have the identity 
$$ T_m' = m^{L(0)} T_m.$$
Especially,
$T_mg_{2r}\in Q_{2r}$.
\end{thm}
\begin{proof} We have already explained that $T'_m$ acts on $M(V)$.\ Our proof of the displayed formula will then imply that $T_m$ also acts.\ Let $f\in M_{2k}(V)$.\ 
Using the definitions we have
\begin{align*}
T'_m f(\tau)&\df m^{k-1}\sum_{[\alpha]} f||_{2k}\alpha(\tau)\\
&=m^{k-1}\sum_{[\alpha]}j(\alpha, \tau)^{-2k}K(\alpha, \tau)^{-1} (\det\alpha)^kf(\alpha\tau)\\
&=m^{2k-1}\sum_{[\alpha]}d^{-2k}K(\alpha, \tau)^{-1}f(\alpha\tau).
\end{align*}
Because $c=0$ for each $\alpha$ appearing in the formula above, then $K(\alpha, \tau)=j^{-2L(0)}(\det\alpha)^{L(0)}=\left(\frac{a}{d}\right)^{L(0)}$,
whence
\begin{align*}
& T'_m f=m^{2k-1} \sum_{[\alpha]} d^{-2k}\left(\frac{d}{a}\right)^{L(0)}f(\alpha\tau)\\
& = m^{2k-1}\sum_{a, d}  \left(\frac{d}{a}\right)^{L(0)} d^{-2k} 
 \sum_b f((a\tau+b)/d).
\end{align*}
Each form $f \in M_k(V)$ is a sum of terms $g_{2r}\otimes v_{\ell}$ where
$g_{2r}\in Q_{2r}$ and $v_{\ell}\in V_{\ell}$ and $r+\ell=k$.\ Thus
$$f((a\tau+b)/d)=\sum_* g_{2r}((a\tau+b)/d)\otimes v_{\ell}$$
and we can apply standard arguments
using the $q$-expansion of $g_{2r}(\tau)\df\sum_n a^g_nq^n$.\ Thus
$$\sum_b g_{2r}((a\tau+b)/d)=\sum a^g_n e^{2\pi ian\tau/d}\sum_{b=0}^{d-1} e^{2\pi i bn/d}.$$
The innermost sum vanishes except when $d|n$, when it is equal to $d$.\ Thus
$$\sum_b g_{2r}((a\tau+b)/d)=\sum_{d|n}a^g_n d q^{an/d}.$$

Now we have
\begin{align*}
 T'_mf &= m^{2k-1}\sum_{g_{2r}\otimes v_{\ell}}\sum_{a, d} \left(\frac{d}{a}\right)^{L(0)} d^{1-2k}     \sum_{d|n}a^{g_{2r}}_n q^{an/d}\otimes v_{\ell}\\
&=\sum_{g_{2r}\otimes v_{\ell}}\sum_{a\mid m} \left(\frac{m}{a^2}\right)^{L(0)}\left\{ a^{2k-1} \sum_{t} a_{mt/a}^g q^{at}\right\}\otimes v_{\ell}\\
&=\sum_{g_{2r}\otimes v_{\ell}}\sum_n \sum_{a\mid (m, n)} \left(\frac{m}{a^2}\right)^{L(0)}\left\{ a^{2k-1} a_{mn/a^2}^g q^{n}\right\}\otimes v_{\ell}\\
&=\sum_{g_{2r}\otimes v_{\ell}}\sum_n \sum_{a\mid (m, n)} \left(\frac{m}{a^2}\right)^{k-r}\left\{ a^{2k-1} a_{mn/a^2}^g q^{n}\right\}\otimes v_{\ell}\\
&=\sum_{g_{2r}\otimes v_{\ell}}\sum_n \sum_{a\mid (m, n)} m^{L(0)}\left\{ a^{2r-1} a_{mn/a^2}^g q^{n}\right\}\otimes v_{\ell}\\
&=m^{L(0)}\sum_{g_{2r}\otimes v_{\ell}}
T_mg_{2r}\otimes v_{\ell}\\
&=m^{L(0)}T_mf.
\end{align*}

This completes the proof of the Theorem.
\end{proof}

\subsection{An alternate Hecke action}Recall the operator $P$ introduced in Subsection \ref{SSP}.\ Consider
    \begin{equation}\label{diagcomm}
    \xymatrix{
M\otimes V^{(2)}\ar[rr]^P\ar[d]_{m^{L(0)}T_m}&& M(V)\ar[d]^{T'_m}\\
M\otimes V^{(2)}\ar[rr]_P && M(V)
    }
    \end{equation}
    
    \begin{thm}\label{thmcommdiag}
    Diagram \eqref{diagcomm} commutes.
    \end{thm}
    \begin{proof}
        Given $f\in M_{2k}(V)$ we only have to check that the \emph{modular parts} (the pieces of depth $0$) of $T'_mf$ and $Pm^{L(0)}T_mP^{-1}f$ are equal.

        Let $f=\sum_n E_2^n\sum_{\ell}^{k-n} f_{n, \ell}\otimes v_{n, \ell}\in M_{2k}(V)$ and let $f^0$ be the modular part of $f$, i.e.,
    $$f^0:=\sum f_{0, \ell}\otimes v_{0, \ell}\in M\otimes V^{(2)}.$$\ Recall that the notation is supposed to imply that
    $f_{n, \ell}\in M_{2\ell}$, $v_{n, \ell}\in V_{k-n-\ell}$. By definition of $P$ we have $P(f^0)=f$.\ Therefore
    $$Pm^{L(0)}T_mP^{-1}f=Pm^{L(0)}T_mf^0$$
    and this has modular part equal to $m^{L(0)}T_mf^0$.\ But this is the modular part of $T'_mf$ by Theorem \ref{thmHecke}.\ This completes the proof of the present Theorem.
    \end{proof}
    
The description of $T'_m$ given in Theorem \ref{thmcommdiag} has some advantages over that of Theorem \ref{thmHecke}.\ That's because we can read-off relations among the $T_m'$ from those of the action of $T_m$ on modular forms (as opposed to quasi-modular forms).\ For example
\begin{cor}
 The operators $T'_m\ (m\geq 1)$ acting on $M_{2k}(V)$ \textit{commute}.\ In fact they realize an action of the Hecke algebra $\mathcal{H}$ and we have
 $$T'_m T'_n = T'_{mn}\ \mbox{if}\ (m, n)=1.$$
 $\hfill\Box$
\end{cor}

Similarly,
 \begin{lem}
    Acting on $M_{2k}(V)$ and for any prime $p$,  we have 
    $$ T'_pT'_{p^n} =T'_{p^{n+1}}+p^{2k-1}T'_{p^{n-1}}.$$
 \end{lem}
 \begin{proof} It is well-known that the usual Hecke operators $T_{p^j}$ acting on $M_{2k}$ satisfy the same relation.\ Thus if $f=g_{2r}\otimes v_{\ell}\in M\otimes V^{(2)}$ with $g\in M_{2r}, v\in V_{\ell}$ and $k=r+\ell$ as before we have 
 \begin{align*} 
 (p^{n+1})^{L(0)}T_p T_{p^n}f &= (p^{n+1})^{L(0)}T_{p^{n+1}}f+ p^{2r-1}p^{2L(0)}(p^{n-1})^{L(0)}T_{p^{n-1}}f\\
 &=(p^{n+1})^{L(0)}T_{p^{n+1}}f+ p^{2k-1}(p^{n-1})^{L(0)}T_{p^{n-1}}f
 \end{align*}
 Now conjugate this relation by $P$
 and use Theorem \ref{thmcommdiag})
 to complete the proof of the Lemma.
\end{proof}

By a standard argument the last two results give us the Euler product representation of the formal Dirichlet series
\begin{equation}
\label{Dseries}
 \sum_{n=1}^{\infty} \frac{T'_n}{n^s} =
 \prod_p \left( 1 - \frac{T'_p}{p^s}+\frac{p^{2k-1}}{p^{2s}}\right)^{-1}.
\end{equation}

Recall that we have the connection $\nabla\colon M_k(V)\rightarrow M_{k+2}(V)$. On other hand, we know that $M(V)\cong M\otimes V^{(2)}$ via the isomorphism $P$. It is natural to determine the action of the connection on $M\otimes V^{(2)}$. More precisely, we wish to describe the map $\nabla_{k}^\prime$ that fits into the diagram

\begin{equation}
    \xymatrix{
M\otimes V^{(2)}\ar[rr]^P\ar[d]_{\nabla_{k}^\prime}&& M(V)\ar[d]^{\nabla_{k}}\\
M\otimes V^{(2)}\ar[rr]_P && M(V)
    }
    \end{equation}

    Let $D_k\colon M_k\rightarrow M_{k+2}$ be the Serre derivative. Recall that $D_k(F)$ consists precisely of the terms in $\theta(F)$ containing no $E_2$ terms. On the other hand $P$ identifies a $V$-valued modular form $F$ as $F=\exp(\mathfrak{c}E_2L(1))F_0$ where $\mathfrak{c}$ is an explicit constant and $F_0=F\mod E_2$ is the term of $F$ with no $E_2$ term. We can now compute the connection as a function on $M\otimes V^{(2)}$.

    \begin{thm}
    With notation as described above, we have
   \[\nabla_{k}^\prime=D_{k}\otimes\textnormal{id}-\mathfrak{c}\frac{E_4}{12}\otimes L(1)+\textnormal{id}\otimes L(-1) \]

    \end{thm}
    \begin{proof}
By our remarks above, to compute $\nabla_{k}^\prime$ it suffices to compute $\nabla_{k}(F)\mod E_2$. In terms of formulae we have

\[\nabla_{k}(F)=\theta(F)-\frac{k}{12}E_2F+L(-1)F.\]

It is immediate that $-\frac{k}{12}E_2$ does not contribute to $F\mod E_2$ and $L(-1)F\mod E_2=L(-1)(F\mod E_2)$. Therefore, $\theta(F)$ is where all the action happens; it suffices to compute $\theta(F)\mod E_2$ when $F$ is a quasi-modular form. Using the Ramanujan identities \ref{Ramanujanids} and the Leibniz rule, it follows that
\begin{align*}
\theta(E_2^i)\mod E_2&=\begin{cases}-\frac{E_4}{12} &i=1\\
0 &i>1\end{cases}
\\
\theta(E_4^i)\mod E_2 &=-\frac{i}{3}E_4^{i-1}E_6\\
\theta(E_6^i)\mod E_2&=-\frac{i}{2}E_6^{i-1}E_4^2\\
\theta(E_4^iE_6^j)\mod E_2&=-\left(\frac{j}{2}E_4^{i+2}E_6^{j-1}+\frac{i}{3}E_4^{i-1}E_6^{j+1}\right)
\end{align*}
These formulae characterize the Serre derivative. From this, if $F=\sum_{i=0}^pE_2^iF_i$ then 
\[\theta(F)\mod E_2=\frac{-F_1E_4}{12}+D_{k}(F_0)\]
Putting this together gives the following: if $F=\sum_{\ell}F_\ell\otimes v_\ell$ in $M\otimes V^{(2)}$ then 
\[P(F)=\exp(\mathfrak{c}E_2L(1))F=\sum_{i\geq 0}\frac{\mathfrak{c}^i}{i!}(E_2^i\otimes L(1))\sum_\ell F_\ell\otimes v_\ell\]
The contribution of $\theta$ to the constant term of $\nabla_{k}(P(F))$ is then $(-\mathfrak{c}\frac{E_4}{12}L(1))F+(D_{k}(F))$. In conclusion

\[\nabla_{k}^\prime=D_{k}\otimes\textnormal{id}-\mathfrak{c}\frac{E_4}{12}\otimes L(1)+\textnormal{id}\otimes L(-1) \]
\end{proof}

The following result is analogous to what was observed for quasi-modular forms in the appendix, lifted to VOA-valued modular forms:
\begin{lem}
    Let $f \in Q(V)$ of weight $k$ be an eigenform for $T_{m}'$ with eigenvalue $\lambda_{m}$ for all $m$. Then $\nabla_{k} (f) = 2\pi i \theta (f) + L(-1)f$ is an eigenform for $T'_{m}$ with eigenvalue $m \lambda_{m}$ for all $m$.
\end{lem}
\begin{proof}
    As in Theorem \ref{t:quasiE2}, let us write
    \begin{align*}
        f = \sum_{i+j=k} \alpha_{i} \beta_{j}
    \end{align*}
    where $\alpha_{i} \in Q_{i}$ and $\beta_{j}$ has $L(0)$-weight $(k-i)/2$. Since $T_{m}(f) = m^{-L(0)} \lambda_{m} f$ this means
    \begin{align*}
        T_{m}(\alpha_{i}) = m^{(i-k)/2} \lambda_{m} \alpha_{i}.
    \end{align*}
    Recall that if $g$ is an eigenform for $T_{m}$ with eigenvalue $\mu_{m}$, then $\theta (g)$ is an eigenform with eigenvalue $m \mu_{m}$. This fact is proven in Theorem \ref{t:quasimodulareigenforms} in the appendix. Using this, we get
    \begin{align*}
        T_{m}'(\theta (f)) = m^{L(0)} \sum_{i+j=k} m^{(i-k)/2 + 1} \lambda_{m} \theta (\alpha_{i} ) \beta_{j} = m \lambda_{m} \theta (f).
    \end{align*}
    Next, we have
    \begin{align*}
        T_{m}'(L(-1)f) = \sum_{i+j=k} m^{(i-k)/2} \lambda_{m} m^{(k-i)/2 + 1} \alpha_{i} L(-1) \beta_{i} = m \lambda_{m} L(-1)f.
    \end{align*}
    Together, the above two displays give us that
    \begin{align*}
        T_{m}'(\theta (f) + \tfrac{1}{2 \pi i} L(-1) f) = m \lambda_{m} \theta (f) + \tfrac{1}{2 \pi i} m \lambda_{m} L(-1)f = m \lambda_{m} ( \theta (f) + \tfrac{1}{2 \pi i} L(-1)f)
    \end{align*}
    hence $T_{m}'(\nabla_{k} (f)) = m \lambda_{m} \nabla_{k} (f)$ which is what we wanted to show.
\end{proof}

\subsection{Hecke eigenstates} 
If $\{f\}$ is a complete set of simultaneous eigenstates for the actions
of $m^{L(0)}T_m$ on $M\otimes V^{(2)}$
then by Theorem \ref{thmcommdiag} $\{Pf\}$ is a complete set of simultaneous eigenstates for the actions of $T'_m$ on $M(V)$.\ We refer to such gadgets as \emph{Hecke eigenstates}.

Based on what we have proved it is easy to write down such a set $\{f\}$.\ The most straightforward case is
$$1\otimes V_k.$$
$m^{L(0)}T_m$ acts as a scalar matrix
with eigenvalue $m^{k-1}\sigma_1(m)$
(cf.\ Remark \ref{rmkTm}).\ The corresponding  Hecke eigenstates in $M(V)$ are, for $v\in V_k$,
$$P(1\otimes v)=\sum_{n \geq 0} \frac{1}{n!}\left( 
\frac{-12}{2\pi i}\right)^n E_2^n L(1)^nv.$$

\medskip
The generic case 
$$M_{2r}\otimes V_{\ell}\ \  (r+\ell=k, r \geq 2)$$
is similar but requires the classical theory of Hecke operators for $M_{2r}$.
Recall that for $2r\geq 4$, $M_{2r}$ contains a unique basis of simultaneous normalized eigenforms for the classical Hecke operators $T_m$.\ Let this basis be denoted by $e^1, \hdots, e^d$ say, and  let the $q$-expansion of $e^i$ be
$$e^i = \sum_{n\geq0} e^i_nq^n.$$
"Normalized" means that $e^i_1=1$.\ Such  a choice is always possible, and we have
$T_m(e^i)=e^i_m e^i$.

\medskip
Choose any basis $u^1, \hdots. u^f$ of $V_{\ell}$.\ Then 
$$ \{ e^i\otimes u^j\}$$
is a complete set of simultaneous eigenstates
of weight $2k$ in $M_{2r}\otimes V_{\ell}$ for the operators $m^{L(0)}T_m$ and 
$$ m^{L(0)}T_m(e^i\otimes u^j) = m^{\ell}e^i_m e^i\otimes u^j.$$

The corresponding Hecke eigenstate in $M_{2k}(V)$ is
$$P(e^i\otimes u^j) = \sum_{n\geq0} \frac{1}{n!}\left(\frac{-2\pi i}{12}\right)^n E_2^ne^i(\tau) L(1)^nu^j.$$

\medskip
From the formal Euler product 
\eqref{Dseries} we can, as usual, deduce the Euler product for the eigenvalues of the above  Hecke eigenstates.\ For example for $P(1\otimes v)$ as above we obtain
$$\sum_{m=1}^{\infty} \frac{m^{k-1}\sigma_1(m)}{m^s}=\prod_p \left(1-\frac{p^{k-1}\sigma_1(p)}{p^s}+\frac{p^{2k-1}}{p^{2s}}\right)^{-1}=\zeta(s-k)\zeta(s-k+1).$$
This is the $L$-function for $E_2(\tau)$ evaluated at $s-k+1$.

\medskip
Similarly, for $P(e^i\otimes u^j)$ we get
$$\sum_{m=1}^{\infty} \frac{m^{\ell}e^i_m}{m^s}=\prod_p \left(1-\frac{p^{\ell}e^i_p}{p^s}+\frac{p^{2k-1}}{p^{2s}}\right)^{-1}.$$
This is just the $L$-function for $e^i(\tau)$ evaluated at $s-\ell$, because we know that
$$\sum_{m=1}^{\infty} \frac{e^i_m}{m^s}=\prod_p \left(1-\frac{e^i_p}{p^s}+\frac{p^{2r-1}}{p^{2s}}\right)^{-1}.$$
Thus, we have reduced the study of Hecke eigensystems on $M(V)$ and $Q(V)$, and their associated $L$-functions, to the classical study of Hecke eigensystems on $M$ and $Q$.

\appendix
\section{Scalar-valued quasi-modular forms}\label{AppA}

For background on the theory of modular forms, we refer the reader to \cite{Zagier123}, \cite{DiamondShurman}.\ We use the notation of 
Section \ref{SNotation} without further comment except that we remind the reader that
$$M:=\CC[E_4, E_6]$$
is the $2\ZZ$-graded algebra of holomorphic modular forms. 

\medskip
We turn to a brief discussion of a quasimodular form, cf.\ \cite{RoyerQuasi}, \cite{Zagier123} for further background.\ A \emph{quasi-modular function} of weight $k$ and depth at most $s$ is a holomorphic map
$f\colon \uhp \to \CC$
 such that 
 \[
   f|_k\gamma(\tau)=\sum_{n=0}^s X^nQ_n(f)
 \]
 for some holomorphic maps $Q_n(f) \colon \uhp \to \CC$ and all $\gamma\in\Gamma$.\ It follows from the definition that $f=Q_0(f)$,
 $f=0$ if $k$ is odd or $k<0$, and $f(\tau\pm 1)=f(\tau)$.\ In particular $f$ has a $q$-expansion.\ As usual, we call $f$ a  holomorphic quasi-modular form if its $q$-expansion at the cusp has no pole (though we are not strict with this nomenclature.)

\medskip
Let $Q_{2k}$ be the space of quasi-modular forms of weight $2k$.\ Then 
$$Q:=\oplus_{k\geq0} Q_{2k}$$
is a $2\ZZ$-graded algebra.\ For example $Q_2$ is spanned by $E_2$ and in fact we have
    \[
    Q = \CC[E_2,E_4,E_6].
    \]

A significant feature of $Q$, fundamental for the present paper, is that $Q$ is closed under the derivative operator
$\theta=qd/dq=(1/2\pi i)d/d\tau$.\ Indeed $\theta$ is a derivation of $Q$ and raises weights by $2$.\ For example
Ramanujan famously proved the identities
\begin{eqnarray}
  &&  \theta (E_{2}) = (E_2^2-E_4)/12,\notag\\
   && \theta (E_{4}) = (E_2E_4-E_6)/3, \label{Ramanujanids}\\
   && \theta (E_{6}) = (E_2E_6-E_4^2)/2. \notag
\end{eqnarray}

It is well-known that the \emph{modular derivative}, a modification of $\theta$
defined by
\begin{align*}
    D f :=\theta f - \frac{k}{12} E_{2} f
    \ \ (f\in M_{k})
\end{align*}
is a derivation of $M$ that raises weights by $2$.

\begin{thm}
        \label{t:quasimodulareigenforms}
        Each space $Q_k$ of quasi-modular forms of weight-$k$ has a basis of eigenforms that can be described as follows: for each $2 < \ell \leq k$, let $\cB_{\ell}$ denote a basis of eigenforms for $M_\ell$, and let $\cB_2 = \{E_2\}$. Then
        \[
        \cE_k \df \bigcup_{0\leq \ell \leq k/2} \theta^{\ell} ( \cB_{k-2\ell} )
        \]
        is a basis of eigenforms for $Q_k$.
    \end{thm}
    \begin{proof}
        
        First we show if $f$ is a quasi-modular eigenform of weight $k$, then $\theta (f)$ is a quasi-modular eigenform of weight $k+2$. First, if $f=\sum_{n\geq 0}a_nq^n \in Q_{k}$, then we have the expression $\theta (f)=\sum_{n\geq 0}na_nq^n \in Q_{k+2}$. Now, $T_m(\theta (f))=\sum_{n\geq 0}b_nq^n$, where the coefficients are given by
        \begin{align*}
        b_n=\displaystyle\sum_{\ell|gcd(m,n)}\ell^{(k+2)-1}(mn/\ell^2)a_{mn/\ell^2}=mn\displaystyle\sum_{\ell|gcd(m,n)}\ell^{k-1}a_{mn/\ell^2},
        \end{align*}
        so that
        \begin{align*}
        T_m(\theta (f))=\displaystyle\sum_{n\geq 0}mn(\sum_{\ell|gcd(m,n)}\ell^{k-1}a_{mn/\ell^2})q^n=m\sum_{n\geq 0}n(\sum_{\ell|gcd(m,n)}\ell^{k-1}a_{mn/\ell^2})q^n.
        \end{align*}
        In comparison, we can easily compute the expressions
        \begin{align*}
        T_m(f)&=\displaystyle\sum_{n\geq 0}(\sum_{\ell|gcd(m,n)}\ell^{k-1}a_{mn/\ell^2})q^n \\
        \theta( T_m(f))&=\displaystyle\sum_{n\geq 0}n(\sum_{\ell|gcd(m,n)}\ell^{k-1}a_{mn/\ell^2})q^n
        \end{align*}
        which makes it clear that $T_m(\theta (f))=m\theta(T_m(f))$. Now suppose $f=\sum_{n\geq 0}a_nq^n$ is an eigenform. Then $T_{m}(f) = \lambda_{m}f$ for some $\lambda_{m} \in \CC$ for all $m \geq 1$. From our computations, above, we see that
        \begin{align*}
            T_m(\theta (f))=m\theta(T_m(f))=m\theta(\lambda_mf)=m\lambda_m(\theta (f))
        \end{align*}
        hence $\theta (f)$ is an eigenform with eigenvalues $m\lambda_m$. Now we show that $\cE_k$ is a linearly independent set of vectors in $Q_k$. Suppose $f = \sum_{i \geq 0} c_{n}q^{n} \in Q_{k}$ for some $k \geq 2$ and that $\theta (f) = 0$. Expanding this relation gives $\sum_{i \geq 0} i \cdot c_{i}q^{i} = 0$ which means that $i \cdot c_{i} = 0$ for all $i \geq 0$. Since $f \in Q_{k}$ for $k \geq 2$, we cannot have that $i=0$ for all $i \geq 0$. We must then have that $c_{i} = 0$ for all $i \geq 0$, implying $\theta : Q_{k} \to Q_{k+2}$ is an injective linear map.

        Now for any $k \geq 2$, $\cB_{k} \subset M_{k}$ denotes a basis of eigenforms for $M_{k}$ and is hence a linearly independent set. The injectivity of $\theta$ shows that $\theta ( \cB_{k}) \subset Q_{k+2}$ is also a linearly independent set. Likewise for any $0 \leq \ell \leq k/2$, the set $\theta^{\ell} ( \cB_{k-2 \ell}) \subset Q_{k}$ is also linearly independent since the composition $\theta^{\ell}$ also gives an injective mapping. It remains to show that the union of such sets
        \begin{align*}
            \bigcup_{0 \leq \ell \leq k/2} \theta^{\ell} ( \cB_{k-2 \ell})
        \end{align*}
        is linearly independent inside $Q_{k}$. This is equivalent to showing that for any $0 \leq \ell_{1}, \ell_{2} \leq k/2$ with $\ell_{1} \neq \ell_{2}$, we have 
        \begin{align*}
            \theta^{\ell_{1}}(\cB_{k-2 \ell_{1}}) \cap \theta^{\ell_{2}}(\cB_{k-2 \ell_{2}}) = \emptyset.         
        \end{align*}
        Suppose $f \in \theta^{\ell_{1}}(\cB_{k-2 \ell_{1}}) \cap \theta^{\ell_{2}}(\cB_{k-2 \ell_{2}})$. By definition, this implies that there exists $g_{1} \in \cB_{k-2 \ell_{1}}$ and $g_{2} \in \cB_{k-2 \ell_{2}}$ with $f = \theta^{\ell_{1}}(g_{1}) = \theta^{\ell_{2}}(g_{2})$. Assuming $\ell_{1} < \ell_{2}$ without loss of generality, we have that $g_{1}$ and $\theta^{\ell_{2} - \ell_{1}}(g_{2})$ are both sent to $f$ under the injective map $\theta^{\ell_{1}}$. But $g_{1} \neq \theta^{\ell_{2} - \ell_{1}}(g_{2})$ since $\theta^{\ell_{2} - \ell_{1}}(g_{2})$ is quasi-modular however $g_{1}$ is strictly modular. Since $\theta^{\ell_{1}}$ is injective, we must then have that $\theta^{\ell_{1}}(g_{1}) = \theta^{\ell_{2}}(g_{2}) = 0$ that is, $f=0$. Thus $\cE_k$ is a linearly independent set of vectors in $Q_k$.
            
        It remains to show that the size of $\cE_k$ is equal to $\dim Q_k$. We can express
        \begin{align*}
            Q_{k} = \bigoplus_{i=0}^{k/2} E_{2}^{i}M_{k-2i},
        \end{align*}
        and this immediately shows that
        \begin{align*}
            \dim Q_{k} = \sum_{i=0}^{k/2} \dim M_{k-2i} = \vert \cE_{k} \vert.
        \end{align*}
        Thus $\cE_k$ indeed forms an eigenbasis for $Q_k$.
        \end{proof}
\section{Vertex algebras}\label{AppB}
There are many variants of the notion of \textit{vertex algebra}.\ Here we will describe just those that we need in this paper.\ They are mainly vertex algebras over $\CC$.\ For a general presentation of vertex algebras over any commutative base ring, see \cite{Mason1}.

\subsubsection{Fields on a linear space}
Fix a $\CC$-linear space $V$, often referred to as (the) \textit{Fock space}.\ Vectors in $V$ are frequently called \textit{states}.\ We consider the linear space $End(V)[[z, z^{-1}]]$ consisting of (doubly infinite)
power series 
$$a(z)=\sum_{n\in \ZZ} a(n)z^{-n-1}$$
where $a(n)$ are arbitrary endomorphisms  of $V$.\ The space of fields on $V$ is the linear subspace of $End(V)[[z, z^{-1}]]$ defined by
$$\mathcal{F}(V):= \{a(z) \mid \forall b\in V\ \mbox{we have}\  a(n)b=0\ \mbox{for}\ n\gg0\}.$$
Thus for a field $a(z)\in \mathcal{F}(V)$ we have
$$ a(z)b :=\sum_n a(n)bz^{-n-1}\in \CC[z^{-1}][[z]].$$

\subsubsection{Creative fields}
We fix once and for all a nonzero state $\mathbf{1}\in V$ called the \textit{vacuum}.\ A field $a(z)\in\mathcal{F}(V)$ is called \textit{creative} if it satisfies
$$a(n)\mathbf{1}=0\ \mbox{for all}\ n\geq 0.$$
Put another way,
$$a(z).\mathbf{1}= a(-1)\mathbf{1}+a(-2)\mathbf{1}z+\hdots \in V[[z]].$$
In this situation we also say that $a(z)$ \textit{creates} the state $a(-1)\mathbf{1}$, that is, the constant term of $a(z).\mathbf{1}$.

\subsubsection{Vertex algebras}
A vertex algebra over $\CC$ is a triple $(V, Y, \mathbf{1})$ consisting of the following ingredients:\\
1.\   $(V, \mathbf{1})$ is a linear space with vacuum state $\mathbf{1}$.\\
2.\  $Y: V\rightarrow \mathcal{F}(V),\ a\mapsto Y(a, z)=\sum_n a(n)z^{-n-1}$ is a linear map such that\\
3.\ Each field $Y(a, z)$ is creative and creates $a\in V$, i.e., $a(-1)\mathbf{1}=a$.\\
4.\ The following structural (Jacobi)  identity holds for all states $a, b, c\in V$ and all integers $r, s, t$:
\begin{eqnarray*}
&&\sum_{i\geq0}\binom{r}{i} ((a(t+i)b)(r+s-i)c= \\
&&\sum_{i\geq0}(-1)^i \binom{t}{i} \left\{ a(r+t-i)b(s+i)c-(-1)^tb(s+t-i)a(r+i)c  \right\} .
\end{eqnarray*}
It follows from these axioms (loc.\ cit.) that $Y(\mathbf{1}, z)=Id_V$, i.e., $\mathbf{1}(n)=\delta_{n, -1}Id_V$.\ 
\subsubsection{The canonical derivation}\label{SScanonical}
Let $V= (V, Y, \mathbf{1})$ be a vertex algebra over $\CC$.\ The creativity axiom permits us to define a sequence $(D_0, D_1, D_2, \hdots)$ of endomorphisms $D_n$ of $V$:
$$\sum_{n\geq0} D_n(a)z^n :=Y(a, z)\mathbf{1}$$
so that
$$D_{n}(a)= a(-n-1)\mathbf{1}\ (a\in V).$$
Thus $D_0=Id_V$.\ We usually write $D=D_1$, that is $Da:=a(-2)\mathbf{1}$.\ It can be shown (loc.\ cit.) that $D$ is a \textit{derivation} of $V$ in the sense that
for every $n$ we have
\begin{eqnarray}\label{derdef}
D(a(n)b)= (Da)(n)b+a(n)Db.
\end{eqnarray}
 Then  $D_n\mathbf{1}=\delta_{n, 1}\mathbf{1}$ and $D_n=\frac{1}{n!}D^n$ so that
$$Y(a, z)\mathbf{1}=e^{zD}a.$$
$D$ is called the \textit{canonical derivation} of $V$.\ If we want to identify the derivation then we may also refer to the quadruple $(V, Y,\mathbf{1}, D)$ as a vertex algebra, it being understood that $D$ is the canonical derivation of $V$.

\medskip\noindent
\underline{Example $1$}.\ The following example is well-known, (\textit{loc.\ cit.}, Section 5.2).\ Let $V$ be a commutative, associative $\CC$-algebra with identity $1$ with $D$  \textit{any} derivation of $V$.\ Identify each $a\in V$ with the endomorphism defined by multiplication by $a$.\ Define the vertex operator by
$$Y(a, z)=e^{zD}a,\ \ \mbox{i.e.,}\ a(n)=0\ (n\geq0), a_n=\tfrac{1}{(-n-1)!}D^{-n-1}a\ (n<0).$$
Then $(V, Y, 1, D)$ is a vertex algebra.

\medskip\noindent
\underline{Example $2$}.\ We may take $D=0$ in the above construction.\ What obtains is the easiest example of a vertex algebra, where the vertex operator $Y(a, z)=a$
is independent of $z$ and  $(V, Y, 1, 0)$ is nothing but the $\CC$-algebra $V$.

\subsubsection{$\ZZ$-graded vertex algebras}\label{SSgr}  A vertex algebra $(V, Y, \mathbf{1}, D)$ over $\CC$ is called $\ZZ$-graded if the Fock space carries a $\ZZ$-grading
$$V = \bigoplus_{k\in \ZZ} V_k$$ into subspaces that satisfy the following conditions\\
1).\ $a\in V_k, b\in V_{\ell}\Rightarrow a(n)b\in V_{k+\ell-n-1}$, \\
2).\ $\dim V_k<\infty$ and $V_k=0$ for $k\ll0$.\\
We say that a state $a\in V_k$ has \textit{conformal degree} $k$.\ (Weight $k$ is more traditional but we reserve this name for modular and quasi-modular forms.)

\medskip
In a $\ZZ$-graded vertex algebra $V$ we have $\mathbf{1}\in V_0$, furthermore $D$ is a raising operator of weight $1$:
\begin{eqnarray*}
D: V_k\rightarrow V_{k+1}.
\end{eqnarray*}
$V$ also admits an \textit{Euler operator} $E$ defined by
$$ E.a = ka\ (a\in V_k).$$
$E$ is only a derivation with respect to the $-1$ product.

\medskip
 Let $End(V)^-$ be the Lie algebra consisting of $End(V)$ equipped with the usual bracket of operators.\  $E$ and $D$ span a $2$-dimensional subalgebra of $End(V)^-$ satisfying
$[E, D]=D.$\ 

\begin{rmk}
There are significant vertex algebras which do \textit{not} satisfy  condition 2).\
For example lattice vertex algebras $V_L$ for an \textit{indefinite} even lattice $L$ \cite{LepowskyLi}.
\end{rmk}
\subsubsection{QVOAs}\label{SSQVOA} A quasi-vertex operator algebra (QVOA) over $\CC$ is a quadruple $(V, Y, \mathbf{1}, \rho)$ with the following ingredients:\\
1.\ $(V, Y, \mathbf{1}, D)$ is a $\ZZ$-graded vertex algebra with canonical derivation $D$ and Euler operator $E$.\\
2.\ $\rho: \mathfrak{sl}_2\rightarrow End(V)^-$ is a representation of $\mathfrak{sl}_2$
such that 
$$\stwomat{1/2}{0}{0}{-1/2}\mapsto E, 
\stwomat{0}{1}{0}{0}\mapsto D,
\stwomat{0}{0}{-1}{0}\mapsto \delta.$$
and the following structural identities hold for $a\in V_k$:\\
3.(Translation covariance).\ $(Da)(n)= -na(n-1)=[D, a(n)],$\\
4.\   $[\delta, a(n)] = (\delta a)(n)+(2k-n-2)a(n+1)$.

\medskip
 One checks that $\delta: V_k\rightarrow V_{k-1}$ is a lowering operator of weight $1$.\
  Our axioms for a QVOA are essentially those of \cite{FHL}, Section 2.8.\
 In VOA theory it is traditional to replace
 $E, D, \delta$ by $L(0), L(-1), L(1)$ respectively, so that the bracket relations are $[L(m), L(n)]=(m-n)L(m+n)$ for $m, n\in\{0, \pm 1\}$.\ A state $a\in V$ is called \textit{quasiprimary} if it satisfies $L(1)a=0$

\medskip
A QVOA $V$ is of \textit{CFT-type} if it satisfies $V_0=\CC\mathbf{1}$.\ This is often a very convenient assumption, for example it implies \cite{DMshifted} that $V_n=0$ for all $n<0$.


\subsubsection{VOAs} A vertex operator algebra (VOA) over $\CC$
is a special type of QVOA $(V, Y, \mathbf{1}, \rho)$  with the following additional ingredient:\\
1. \ A distinguished state $\omega\in V$ with vertex operator 
\begin{eqnarray*}
   Y(\omega, z)=: \sum_{n\in\ZZ} L(n)z^{-n-2} 
\end{eqnarray*}
such that the modes of $\omega$ satisfy the Virasoro relations
\begin{eqnarray}\label{Virreln}
  [L(m, L(n)]=(m-n)L(m+n)+\frac{m^3-m}{12}cId_V 
\end{eqnarray} 
for some $c\in\CC$ (the \textit{central charge} of $V$).\\
2.\ The operators $L(n)$ and $Id_V$ span a Lie subalgebra $Vir\subseteq End(V)^-$(a Virasoro algebra of central charge $c$; this follows from (\ref{Virreln})) and $\rho$ is the restriction of $Vir$ to the 
$\mathfrak{sl}_2$-subalgebra spanned by $L(0), L(\pm1)$.\\
This VOA is usually denoted $(V, Y, \mathbf{1}, \omega)$ and $\omega$ is often called the \textit{conformal} or \textit{Virasoro} vector.\ We have $L(-2)\mathbf{1}=\omega, L(0)\omega=2\omega, L(1)\omega=0$.

\subsubsection{Some categories}\label{SScat} It is evident that these vertex-type objects  define concrete categories, in particular we have the category $\mathbf{QV}$ whose objects are
QVOAs and morphisms $f:(U, Y, \mathbf{1}, \rho)\rightarrow (V, Y', \mathbf{1}, \rho')$ are vacuum-preserving linear maps $U\stackrel{f}{\longrightarrow} V$ that preserve
all products ($f(a(n)b)=f(a)(n)f(b)$) and
intertwine the $\mathfrak{sl}_2$-representations ($f\rho(g)=\rho'(g)f,\ g\in\mathfrak{sl}_2$).\ 

\medskip
Similarly, the category $\mathbf{V}$ of VOAs has VOAs as objects with morphisms $f:(U, Y, \mathbf{1}, \omega)\rightarrow (V, Y', \mathbf{1}, \omega')$ being linear maps that preserve vacuum states and conformal vectors and preserve all products as in  $\mathbf{QV}$.

\medskip
Every VOA is a QVOA and there is a forgetful functor
$\mathbf{V}\longrightarrow \mathbf{QV}$
where we ignore the presence of the conformal vector.

\medskip
We can define tensor products of QVOAs (coproduct in $\mathbf{QV}$) 
$$(U, Y, \mathbf{1}, \rho)\otimes(V, Y', \mathbf{1}', \rho'):= (U\otimes V, Y\otimes Y', \mathbf{1}\otimes\mathbf{1}', \rho\otimes \rho').$$
The tensor product  vertex operator is
$$(Y\otimes Y')(a\otimes b, z)=Y(a, z)\otimes Y'(b, z),$$
which means that
$$(a\otimes b)(n)=\sum_{i+j+1=n} a(i)\otimes b(j)$$
(cf. \cite{Mason1}, Subsection 6.5.) and
$\rho\otimes \rho'$ is the tensor product representation of $\mathfrak{sl}_2$ on $U\otimes V$.\ Thus
$$g\mapsto \rho(g)\otimes Id+Id\otimes \rho'(g).$$ 
In particular $(\rho\otimes \rho')(L(0))=L(0\otimes Id+Id\otimes L(0)$ is the Euler operator for $U\otimes V$ and
$U\otimes V$ gets the tensor product grading:
$$(U\otimes V)_n= \sum_{k+\ell=n} U_k \otimes V_{\ell},$$

One checks that the canonical derivation of $U\otimes V$
coincides with the action of $L(-1)$ on $U\otimes V$.
\subsubsection{Doubling}\label{SSdouble} Suppose that 
$(V, Y, \mathbf{1})$ is a $\ZZ$-graded vertex algebra with $V=\oplus_k V_k$.\ The \textit{doubled} space is
$(V^{(2)}, Y, \mathbf{1})$ where
$$V^{(2)}= \oplus_k V^{(2)}_{2k},\ V_{2k}^{(2)}:=V_k.$$

$(V^{(2)}, Y, \mathbf{1})$ is a vertex algebra, however it may no longer be 
$\ZZ$-graded because axiom 1) in Subsection \ref{SSgr} may not be satisfied.
\subsubsection{Rank $1$ Heisenberg VOA}\label{AppS} (cf.\  \cite{Kac}, \cite{LepowskyLi}).\  The complex Heisenberg Lie algebra has a basis
$\{h(n), 0\neq n\in\ZZ\}\cup{\{k\}}$ satisfying the canonical commutator relations
\begin{eqnarray}\label{CCRs}
    [h(m), h(n)]= m\delta_{m, -n}k, \ [h(n), k]=0.
\end{eqnarray}

There is a VOA $(S, Y, \mathbf{1}, \omega)$ for which the Fock space $S$ is a \textit{Verma module} induced from the $1$-dimensional space $\CC\mathbf{1}$ satisfying $h(n).\mathbf{1}=0(n>0), k.\mathbf{1}=\mathbf{1}$.\ $S$ has a basis consisting of states
$$h(-n_1)h(-n_2)\hdots h(-n_k).\mathbf{1} \ (n_1\geq\hdots\geq n_k\geq1).$$
$V_m$ is spanned by those states for which
$\{n_1, \hdots, n_k\} \vdash m$.\ In particular
$S$ is of CFT-type, $S_1$ is spanned by $h:=h(-1)\mathbf{1}$ (the canonical weight $1$ state) and its vertex operator is
$$Y(h, z):= \sum_{n\in\ZZ} h(n)z^{-n-1}$$
where the action of the mode $h(n)$ on $S$ is determined by (\ref{CCRs}).\ The Virasoro vector is
$\omega:=\tfrac{1}{2}h(-1)^2.\mathbf{1}$  and the central charge (or \textit{rank}) is $c=1$;\ $h$ is a quasiprimary state.

\medskip
A common approach to this VOA is to set $x_{-n}=h(-n).\mathbf{1}\ (n>0)$ and identify $V$ with the symmetric algebra $\CC[x_{-1}, x_{-2}, \hdots]$ (whence the appellation $S$).\ Then for
$n>0$ the operator $h(-n)$ is multiplication by $x_{-n}$ and $h(n)=n\partial/\partial x_{-n}$.

\bibliography{refs}
\bibliographystyle{plain}
\end{document}